\documentclass[10pt]{amsart} %

\usepackage{float}

\usepackage{epsfig}
\usepackage{tikz}
\usepackage{array,amsmath, enumerate,  url, psfrag}
\usepackage{amssymb, amsaddr, fullpage}
\usepackage{graphicx,subfigure}
\usepackage{color}
\restylefloat{figure}

\newtheorem{thm}{Theorem}[section]
\newtheorem{lem}[thm]{Lemma}

\newtheorem{cor}[thm]{Corollary}

\newtheorem{definition}[thm]{Definition}
\newtheorem{example}[thm]{Example}

\title{DP-3-coloring of some planar graphs}

\author{  Runrun Liu$^{1}$ \hskip 0.2in Sarah Loeb$^{2}$\hskip 0.2in Yuxue Yin$^{1}$\hskip 0.2in Gexin Yu$^{1,2}$}

\address{
$^{1}$\small School of Mathematics and Statistics, Central China Normal University, Wuhan, Hubei, China.\\
$^2$\small Department of Mathematics, The College of William and Mary, Williamsburg, VA, 23185, USA.
}

\thanks{The research of the last author was supported in part by the Natural Science Foundation of China (11728102) and the NSA grant H98230-16-1-0316.}

\email{gyu@wm.edu}
\date{\today}

\begin{document}
\maketitle

\begin{abstract}
In this article, we use a unified approach to prove several classes of planar graphs are DP-$3$-colorable, which extend the corresponding results on $3$-choosability.
\end{abstract}

\section{Introduction}
Graph coloring is a central topic in graph theory. A \emph{proper $k$-coloring} of a graph $G$ is a function $c$ that assigns an element $c(v) \in [k]$ to each $v \in V(G)$ so that adjacent vertices receive distinct colors. We say that $G$ is \emph{$k$-colorable} if it has a $k$-coloring and call the minimum $k$ such that $G$ has a $k$-coloring the \emph{chromatic number of $G$}, denoted by $\chi(G)$. The famous Four Color Theorem states that every planar graph is properly $4$-colorable.

However, it is NP-complete to decide whether a planar graph is properly $3$-colorable, which provides motivation to look for sufficient conditions for planar graphs to be 3-colorable. For example, Gr\"otzsch~\cite{G59} showed every planar graph without 3-cycles is 3-colorable.

Vizing~\cite{V76},  and independently Erd\H{o}s, Rubin, and Taylor~\cite{ERT79} introduced list coloring as a generalization of proper coloring. A \emph{list assignment} $L$ gives each vertex a list of available colors $L(v)$. A graph $G$ is {\em $L$-colorable} if there is a proper coloring $c$ of $V(G)$ such that $c(v)\in L(v)$ for each $v\in V(G)$. A graph $G$ is {\em $k$-choosable} if $G$ is $L$-colorable for each $L$ with $|L(v)|\ge k$ and the minimum $k$ for which $G$ is $k$-choosable is the \emph{list-chromatic number} $\chi_\ell(G)$.

Since the definition of $k$-choosability requires $G$ to be $L$-colorable for the specific lists $L(v) = [k]$ for all $v \in V(G)$, we have $\chi(G) \le \chi_\ell(G)$. Voigt~\cite{V93} found a non-4-choosable planar graph, showing this inequality may be strict, even if the family of planar graphs.  Thomassen~\cite{T94} showed that every planar graph is $5$-choosable giving an upper bound for this class.

Like for 3-colorability of planar graphs, sufficient conditions for planar graphs to be 3-choosable have been studied. However, the problem is more difficult, Voigt~\cite{V95} gave a planar graph without 3-cycles that is not 3-choosable. On the other hand, Thomassen~\cite{T95} showed that every planar graphs without 3-cycles and 4-cycles are 3-choosable. Some other conditions are included in Theorem~\ref{thm1}.

One difficulty in the study of list coloring problems is that some important techniques used in coloring (for example, identification of vertices) are not feasible in list coloring.  To overcome this difficulty,  Dvo\^{r}\'{a}k and Postle~\cite{DP17} introduced DP-coloring under the name correspondence coloring, as a generalization of list coloring.

\begin{definition}
Let $G$ be a simple graph with $n$ vertices and let $L$ be a list assignment for $V(G)$. For each edge $uv$ in $G$, let $M_{uv}$ be a matching between the sets $L(u)$ and $L(v)$ and let $\mathcal{M}_L = \{ M_{uv} : uv \in E(G)\}$, called the \emph{matching assignment}. Let $H_{L}$ be the graph that satisfies the following conditions
\begin{itemize}
\item each $u\in V(G)$ corresponds to a set of vertices $L(u)$ in $H_L$
\item for all $u \in V(G)$, the set $L(u)$ forms a clique
\item if $uv \in E(G)$, then the edges between $L(u)$ and $L(v)$ are those of $M_{uv}$
\item if $uv \notin E(G)$, then there are no edges between $L(u)$ and $L(v)$
\end{itemize}
If $H_L$ contains an independent set of size $n$, then $G$ has a {\em $\mathcal{M}_L$-coloring}. The graph $G$ is {\em DP-$k$-colorable} if, for any matching assignment $\mathcal{M}_L$ in which $L(u)\supseteq[k]$ for each $u \in V(G)$, it has a $\mathcal{M}_L$-coloring. The minimum $k$ such that $G$ is DP-$k$-colorable is the {\em DP-chromatic number} of $G$, denoted by $\chi_{DP}(G)$.
\end{definition}

As in list coloring, we refer to the elements of $L(v)$ as colors and call the element $i\in L(v)$ chosen in the independent set of an $\mathcal{M}_L$-coloring as the color of $v$.

DP-coloring generalizes list coloring, even with the restriction that $L(u) = [k]$ for all $v \in V(G)$. To see this, consider a list assignment $L'$ with $|L'(u)| = k$ for all $u \in V(G)$. We can biject the elements of $L'(u)$ and $[k]$ and, for each $uv \in E(G)$, let $M_{uv}$ be a matching between the colors of $u$ and $v$ that correspond to equal elements of $L'(u)$ and $L'(v)$. Accounting for relabeling, this $\mathcal{M}_L$-coloring is equivalent to an $L'$-coloring. Thus, $\chi_\ell(G) \le \chi_{DP}(G)$. However, DP-coloring and list coloring can be quite different. For example, Bernshteyn~\cite{B16} showed that the DP-chromatic number of every graph with average degree $d$ is $\Omega(d/\log d)$, while Alon\cite{A00} proved that $\chi_l(G)=\Omega(\log d)$ and the bound is sharp.

Dvo\^{r}\'{a}k and Postle~\cite{DP17} used this notation proved that every planar graph without cycles of lengths from $4$ to $8$ is $3$-choosable (actually a stronger form using DP-coloring), solving a long-standing conjecture of Borodin~\cite{B13}. Since then much attention was drawn on this new coloring, see for example, \cite{B16,B17,BK17a,BK17b,BKP17,BKZ17,KO17a,KO17b,KY17, LLNSY18}.

We are interested in DP-coloring of planar graphs. Dvo\v{r}\'ak and Postle~\cite{DP17} noted that Thomassen's proofs~\cite{T94} for choosability can be used to show $\chi_{DP}(G)\le5$ if $G$ is a planar graph, and $\chi_{DP}(G)\le3$ if $G$ is a planar graph with no 3-cycles and 4-cycles. Some sufficient conditions were given in \cite{KO17a,KO17b,LLNSY18} for a planar graph to be DP-$4$-colorable.

We study the sufficient conditions for a planar graph to be DP-$3$-colorable.  Many such conditions are known for $3$-choosability of planar graphs, some of which are listed in Theorem~\ref{thm1}.

\begin{thm}\label{thm1}
A planar graph $G$ is $3$-choosable if one of the following conditions holds
\begin{itemize}
\item $G$ contains no $\{3,6,7,8\}$-cycles. (\cite{L09}, 2009)
\item $G$ contains no $\{3, 5, 6\}$-cycles. (\cite{LSS05}, 2005)
\item $G$ contains no $\{4,5,6,9\}$-cycles. (\cite{ZW05a}, 2005)
\item $G$ contains no $\{4,5,7,9\}$-cycles. (\cite{ZW04}, 2004)
\item $G$ contains no $\{5,6,7\}$-cycles and the distance of triangles is at least two. (\cite{LCW16}, 2016)
\end{itemize}
\end{thm}

We show that these conditions are also sufficient for a planar graph to be DP-$3$-colorable.

\begin{thm}\label{thmmain}
Every planar graph listed in Theorem~\ref{thm1} is DP-$3$-colorable.
\end{thm}

The proofs of the results use the discharging method, which uses strong induction. Say a structure is \emph{reducible} if it cannot appear in a minimal counterexample $G$. Quickly one may find that all the proofs in Theorem~\ref{thm1} relied heavily on the following fact: an even cycle whose vertices all have degree three is reducible. In other words, if $C$ is an even cycle whose vertices all have degree 3, then any coloring of $G-C$ can be extended to $G$. This follows from the fact that even cycles are 2-choosable. However, this structure is not reducible for results on DP-$3$-coloring, since even cycles may fail to be DP-$2$-colorable, as illustrated in Figure~\ref{not2DP}.

\begin{figure}[ht]
\includegraphics[scale=0.2]{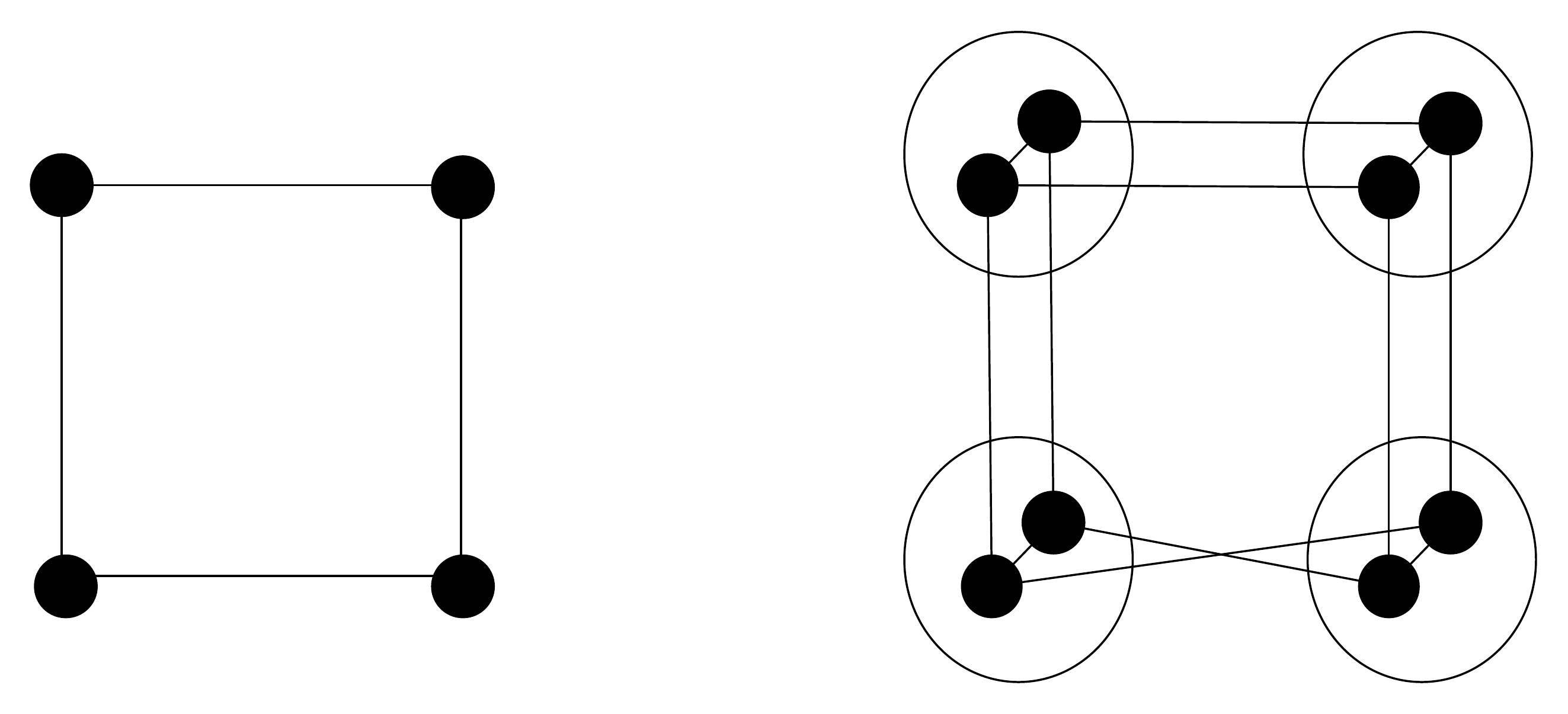}
\caption{A $4$-cycle is not DP-$2$-colorable: the left is a $4$-cycle $G$, and the right is the graph $H_L$.}\label{not2DP}
\end{figure}

Thus new ideas are needed to reach our goal. We will show, by way of discharging, that each of these planar graphs contains a ``near-$(k-1)$-degenerate'' subgraph, which is reducible as shown in Lemma~\ref{near-2-degenerate}. Lemma~\ref{near-2-degenerate} is phrased more generally to give a reducible structure for DP-$k$-coloring; special forms of this structure (namely, theta subgraphs) were the essential components in the proofs of~\cite{KO17a, KO17b, LLNSY18}.

The paper is organized as follows. In section~\ref{reducible}, we provide essential definitions and prove the essential reducible structures needed in all the proofs. In each of the following sections, we give a proof of a part of Theorem~\ref{thmmain}.

 \section{Reducible configurations and a brief introduction to the discharging}\label{reducible}

Graphs mentioned in this paper are all simple.  A $k$-vertex ($k^+$-vertex, $k^-$-vertex) is a vertex of degree $k$ (at least $k$, at most $k$). The \emph{length} of a face is the number of vertices on its boundary, with repetition included.
We may also refer to a $(\ell_1, \ell_2, \ldots, \ell_k)$-face is a $k$-face with facial walk $v_1,v_2,\ldots,v_k$ such that $d(v_i)=\ell_i$.

 \subsection{Reducible configurations}

\begin{lem}\label{minimum}
Let $G$ be a smallest graph (with respect to the number of vertices) that is not DP-$k$-colorable. Then $\delta(G)\ge k$.
\end{lem}

\begin{proof}
Let $v$ be a vertex with $d(v)<k$. Any $\mathcal{M}_L$-coloring of $G-v$ can be extended to $G$ since $v$ has at most $d(v)$ elements of $L(v)$ forbidden by the colors selected for the neighbors of $v$, while $|L(v)|=k$.
\end{proof}

Let $H$ be a subgraph of $G$.  For each vertex $v\in H$, let $A(v)$ be the set of vertices in $L(v)$ that are not neighbors of vertices in $\cup_{u\in G-H} L(u)$.  One may think of $A(v)$ as the available colors of $v$ after $G-H$ being colored.

\begin{lem}\label{near-2-degenerate}
Let $k \ge 3$ and $H$ be a subgraph of $G$. If the vertices of $H$ can be ordered as $v_1, v_2, \ldots, v_{\ell}$ such that the following hold
\begin{itemize}
\item[(1)] $v_1v_{\ell}\in E(G)$ and $|A(v_1)|>|A(v_{\ell})|\ge1$,
\item[(2)] $d(v_{\ell})\le k$ and $v_{\ell}$ has at least one neighbor in $G-H$,
\item[(3)] for each $2\le i\le \ell-1$, $v_i$ has at most $k-1$ neighbors in $G[\{v_1, \ldots, v_{i-1}\}]\cup (G-H)$,
\end{itemize}
then a DP-$k$-coloring of $G-H$ can be extended to a DP-$k$-coloring of $G$.
\end{lem}

\begin{proof}
Fix a matching assignment $\mathcal{M}_L$ with $L(v) = [k]$ for all $v \in V(G)$ and fix a $\mathcal{M}_L$-coloring of $G-H$. Since by (2) $v_{\ell}$ has a neighbor in $G-H$, a color $c\in L(v_{\ell})$ is forbidden by its neighbors in $G-H$. Now by (1) we can choose a color $c'\in L(v_1)$ such that $c'$ is matched to $c$ in $M_{v_1v_{\ell}}$ or not matched with anyone in $L(v_{\ell})$ at all. We may then greedily color $v_2,\ldots,v_{\ell}$ in order, since by (3) there is always at least one color available for $v_i$ $(2\le i\le {\ell-1})$ when we get to it. The choice of $c'$ for $v_1$ guarantees that $v_{\ell}$ also has a color available.
\end{proof}

A \emph{$(d_1,\ldots,d_t)$-walk} $u_1\ldots u_t$ on a face $C$ is a set of consecutive vertices along the facial walk of $C$ such that $d(u_i) = d_i$.  We allow for any of the $d_i$ to be replaced by $d_i^+$ or $d_i^-$ to represent that $d(u_i) \ge d_i$ or $d(u_i) \le d_i$ respectively. Now that a face $C$ may contain a cut vertex, so repeated vertices on the walk is possible. If the vertices are distinct, then we also call it a {\em $(d_1,\ldots,d_t)$-path}. When $t=2$, we sometimes call it a {\em $(d_1,d_2)$-edge}. For $1 \le i \le t-1$, we say the edge $u_i u_{i+1}$ \emph{controls} the face that is adjacent to $C$ across $u_i u_{i+1}$. We generalize this to say that a $(d_1,\ldots,d_t)$-walk $u_1\ldots u_t$ \emph{controls} the faces that are adjacent to $C$ over the edges $u_i u_{i+1}$. A $k$-vertex $v$ with $k\ge 4$ on a $7^+$-face $f$ is {\em rich} to $f$ if none of the two edges of $v$ on $f$ control a $4^-$-face,  {\em semi-rich} if exactly one of them controls a $4^-$-face, and {\em poor} if they control two $4^-$-faces.

In our proofs, we will often be concerned with $(d_1,\ldots,d_t)$-walks where every controlled face is a $4^-$-face. For this reason, we say a $(d_1,\ldots,d_t)$-walk is \emph{$d$-controlling} (or \emph{$d^-$-controlling}, or \emph{$d^+$-controlling}) if every face it controls has length $d$ (or at most $d$ or at least $d$ respectively). In addition, we say a $(d_1,\ldots,d_t)$-walk is \emph{maximal} if it is $4^-$-controlling, but the edges of $C$ immediately before and after it control $5^+$-faces.

Let $f$ be a $7^+$-face. We use $s_0(f)$ to denote the number of vertices of $f$ that is not on the $4^-$-controlling walks of $f$, and $s_1(f)$ be the number of semi-rich $4^+$-vertices plus the number of $4^-$-faces with at least two $4^+$-vertices controlled by $(3,4^+)$-edges and the number of $4^-$-faces with at least three $4^+$-vertices controlled by $(4^+,4^+)$-edges.  We also let $t_i(f)$ for $i\ge 0$ be the number of maximal $4^-$-controlling walks with $i$ internal vertices (when each edge of $f$ controls a $4^-$-face, we will regard the number of internal vertices of the walk to be $d(f)-2$).  Then
\begin{equation}\label{s0ti}
s_0(f)+\sum_{i\ge 0} (i+2)t_i(f)=d(f).
\end{equation}

If it is clear from the context, we usually write $s_0, s_1, t_i$ for $s_0(f), s_1(f), t_i(f)$, respectively.

For our main reducible configuration, we are concerned about particular types of $(d_1,\ldots,d_t)$-paths. 

\begin{definition}\label{def:special}
A $(4^-,4, \ldots,4, 3)$-path $u_1,\ldots,u_t$ is \emph{special} if,
\begin{itemize}
\item for each $i < t$, there is a path $P_i=u_{i} \ldots u_{i+1}$ such that all internal vertices of $P_i$ have degree 3, and
\item $V(P_i)\cap V(C)=\{u_i, u_{i+1}\}$ and $V(P_i)\backslash V(C)\not=\emptyset$.
\end{itemize}

\end{definition}

\begin{figure}[ht]
\includegraphics[scale=0.12]{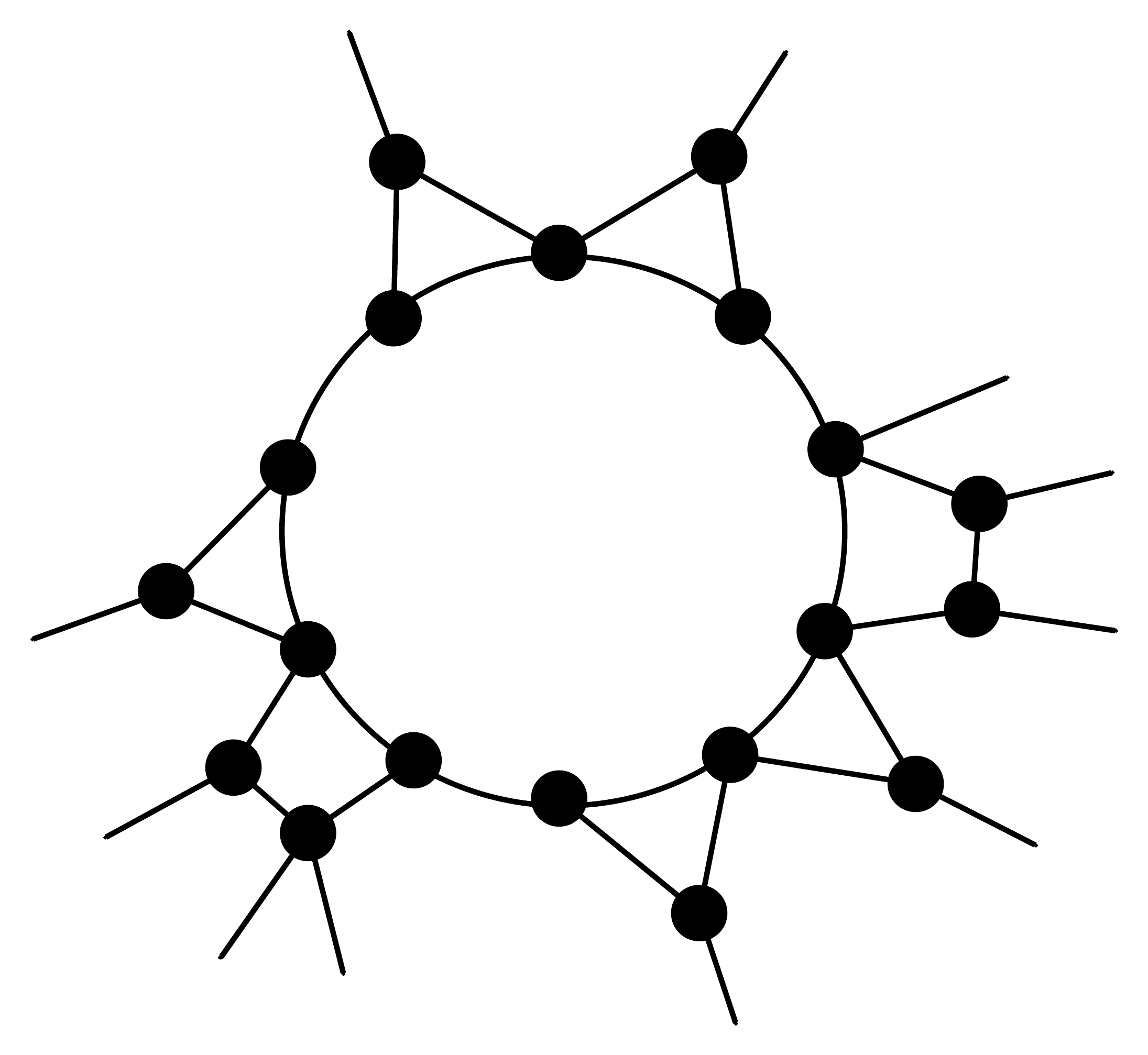}
\caption{A $10$-cycle with a $3$-controlling special $(3,4,3)$-path, a special $(4, 4, 4, 3)$-path, and a maximal $(3, 4, 3)$-path that is not special but has a special $(4,3)$-path using two of its vertices.} \label{fig:special}
\end{figure}

In our proofs, we will generally consider the case that the $P_i$ of a special $(4^-,\ldots,3)$-path $u_1,\ldots,u_t$ that are formed along the boundary of the face controlled by $u_i u_{i+1}$ for $i \le t-1$, but that need not be the case in general.

By applying Lemma~\ref{near-2-degenerate} to one of the subcycles, we find that a cycle of 3-vertices with a chord is reducible as long as at least one vertex has an extenal neighbor. The condition of an external neighbor is required to apply Lemma~\ref{near-2-degenerate}. This condition is also necessary in general, since $K_4$ is a 4-cycle of 3-vertices with two chords.   However, Lemma~\ref{near-2-degenerate} also applies to a larger family of cycles. The chord in a cycle of 3-vertices is extended to allow a special $(3,4,\ldots,4,3)$-path, although we still require that this special path yields a vertex with an external neighbor in a particular way. It also allows for the cycle to have semi-rich 4-vertices as long as these 4-vertices belong to special $(4^-,\ldots,3)$-paths that obey an orientation around the cycle.

\begin{lem}\label{theta-graph-extension}
Assume that $G$ contains no adjacent $4^-$-cycles and is not DP-$3$-colorable such that every proper subgraph of $G$ is DP-$3$-colorable.  Let $f$ be a $7^+$-face bounded by a cycle in $G$ and $f$ contains a special $(3,4,...,4, 3)$-path $P$.  
If for every controlled $4^-$-face $f'$ of $f$,  each vertex of $f'$ not on $f$ has a neighbor outside the subgraph formed by $f$ and the controlled $4^-$-faces of $f$,  then one of the following is true
\begin{enumerate}[(i)]
\item $f$ contains a $5^+$-vertex or a rich $4$-vertex.
\item $s_1(f)\ge 1$, and $s_1(f)=1$ only if $f$ contains a $(3,4,\ldots,4,3)$-path so that exactly one $(4,4)$-edge controls a $(4,4,4^+)$-face or a $(4,4,4^+,3^+)$-face.
\item Furthermore, if all maximal paths other than $P$ are $(3^+,3^+)$-paths, then $f$ contains at least one semi-rich $4$-vertex, and if exactly one, then it is on a $4^-$-face with at least two $4^+$-vertices.
\end{enumerate}
\end{lem}

\begin{proof}
Suppose that none of (i), (ii) and (iii) is true, that is, $G$ contains no $5^+$-vertices or rich $4$-vertices, or $s_1(f)=0$, or $s_1(f)=1$ and each $(4,4)$-edge of $f$ controls a $4^-$-face without other $4^+$-vertices, or all maximal paths except $P$ are $(3,3)$-paths, or all maximal paths except $P$ are $(3,3)$-paths and one special $(3,4)$-path. Then all $4^+$-vertices on $f$ are poor or semi-rich $4$-vertices, and there is at most one semi-rich $4$-vertex, and if $f$ is adjacent to a $4^-$-face $f'$ that is controlled by a $(3,4)$-edge and has one $4^+$-vertex not on $f$, then $f$ contains no semi-rich $4$-vertex, and $f'$ is controlled by a $(3,4,\ldots,3)$-path. In this latter case, we shall partition the path into a $3$-vertex and a special $(4, \ldots, 3)$-path.  So we may assume that $f$ consists of $3$-vertices and special paths with at most one semi-rich $4$-vertex.

Let $f=u_1u_2\ldots u_k$. We create an ordering of vertices on $f$ and on some of the controlled faces by edges of $f$ as follows:
\begin{itemize}
\item The initial list is $u_1, u_2, \ldots, u_k$ so that $u_ku_1$ is a $(3,4^-)$-edge on a special $(3,\ldots,3)$-path.
\item For each $i\in [k-1]$, let $P_i=u_i\ldots u_{i+1}$ on the $4^-$-face controlled by $u_iu_{i+1}$, then we insert the vertices of $P_i$ to $S$ following $u_i$, unless $u_iu_{i+1}$ is on a $(3,3)$-path.
\item Let $P_k$ be the path from $u_k$ to $u_1$ on the controlled $4^-$-face by $u_ku_1$. Insert the vertices of $P_k$ to $S$ following $u_k$.

\item We assume that the semi-rich $4$-vertices of the special $(4,\ldots, 3)$-path is the lowest-indexed vertex among the vertices of the path in the list. (Note that this is valid since  $u_ku_1$ may be selected in opposite orderings.)
\end{itemize}

We shall denote the final list as $S$.  Observe that $S$ has no repeated vertices, $u_1$ has no neighbor outside of $S$, and the last vertex has a neighbor outside of $S$.
By the construction that each vertex in $S-V(f)$ has degree $3$, each vertex in $S$ has at most two neighbors in $G-S$ and in earlier vertices in $S$. So by Lemma~\ref{near-2-degenerate}, a DP-3-coloring of $G-S$ can be extended to $G$, a contradiction.
\end{proof}

\subsection{A brief discussion of the discharging}
\
\noindent

Throughout this paper, let $G$ be a minimum counterexample in each of the following sections. We use $\mu(x)$ denote the intial charge of a vertex or face $x$ in $G$ and $\mu^*(x)$ to denote the final charge after the discharging procedure. In all of our proofs, we use $\mu(v) = 2d(v) - 6$ for each vertex $v$ and $\mu(f) = d(f) - 6$ for each face $f$. Then by Euler formula,
$$\sum_{v\in V(G)} (2d(v)-6)+\sum_{f\in F(G)} (d(f)-6)=-12.$$

By Lemma~\ref{minimum}, for DP-$3$-coloring, only faces start with negative charge.  We move charge around and argue that  every vertex and face ends up with non-negative charge. Since charge is only moved, this contradiction proves our conclusion.

Since only $5^-$-faces have negative initial charge, we use the following rule to give them enough charge.

\begin{quote}
(R1) Each $5^-$-face gets $1$ from each $4^+$-vertex on it, and gets its rest evenly through the edges from its adjacent $7^+$-faces.
\end{quote}

Note that a $5^-$-face $f$ may share more than one edge with a $7^+$-face, which in particular could happen when a $7^+$-face contains cut vertices.  In this case, we require that $f$ gets the rest evenly {\em through the edges} from the adjacent $7^+$-faces.

 Now $4^+$-vertices on a $7^+$-face and some $4^-$-faces may  have surplus charge, we will let them give to the $7^+$-faces following somewhat different rules (R2a), (R2b) et al  from one setting to another, which will be specified in each of the subsequent sections. Still, some very particular $7^+$-faces may still have negative charges, and we create to a bank to send them the charges, which is (R3a), (R3b) et al.  Not every setting needs to use the bank though.

\section{Planar graphs without cycles of lengths in $\{3,6,7,8\}$}

As a warm-up, we prove Theorem~\ref{3678}. The discharging rules used in later sections are similar in spirit to the ones used in the proof, although more complicated situations arise.

\begin{thm}\label{3678}
Every planar graph without 3-, 6-, 7-, and 8-cycles is DP-$3$-colorable.
\end{thm}

\begin{proof}
We first observe that if two $5^-$-faces are adjacent, then they must be two $5$-faces sharing a common edge and a $4^+$-vertex, and a $5$-face cannot be adjacent to two $5$-faces.

Note that a $4, 5$ and $6^+$-vertex $v$ is incident with at most two, three and $d(v)$ faces of degree at most five, respectively. By (R1), each vertex has nonnegative final charge, and $5^-$-faces receive enough charge. We just need to show that the $6^+$-faces, which in this case are $9^+$-faces, end with non-negative charge.

Let $f$ be a $9^+$-face. Note that $f$ gives $0$ to each adjacent $5$-face that is adjacent to a $5$-face, since the $5$-face gets at least $1$ from its incident $4^+$-vertices.  Note also that $f$ gives at most $\frac{1}{2}$ to each $5^-$-face controlled by a $(3,3)$-path, $\frac{1}{4}$ to each $5^-$-face controlled by a $(3,4^+)$-path, and nothing to faces controlled by $(4^+,4^+)$-paths. Thus, $f$ gives away at most $\frac{1}{2}$ to the $5^-$-faces controlled by each maximal $5^-$-controlling walk, and by \eqref{s0ti}, $f$ gives at most $\frac{1}{2}\sum_{i}t_i\le \frac{1}{4}d(f)$ to the adjacent $5^-$-faces. So,
$\mu^*(f) \ge d(f) - 6 - \frac{1}{4}d(f) = \frac{3}{4} \left(d(f) - 8 \right)> 0. $
\end{proof}

\section{Planar graphs without $3$-, $5$-, $6$-cycles}\label{356}

\begin{thm} \label{Thm:356} Every planar graph $G$ without 3-, 5-, or 6-cycles is DP-$3$-colorable. \end{thm}

\begin{proof}
Besides (R1), we also use the following rules:

\begin{enumerate}[(R2a)]
\item every $7$-face gains $\frac{1}4$ from each incident semi-rich $4$-vertex,  $\frac{1}{2}$ from each incident rich $4$-vertex or $5^+$-vertex, $\frac{1}{4}$ from each adjacent $(3^+,4^+,4^+,4^+)$-face.
\end{enumerate}

We first check the final charges of vertices by (R1) and (R2a). By Lemma~\ref{minimum}, $\delta(G)\ge3$.  If $d(v)=3$, then $\mu^*(v)=\mu(v)=0$. Note that $v$ is incident with at most $\lfloor\frac{d(v)}{2}\rfloor$ $4$-faces. If $d(v)=4$, then $v$ gives $1$ to each incident $4$-face, $\frac{1}{4}$ to each incident $7$-face when $v$ is semi-rich, $\frac{1}{2}$ to each incident $7$-face when $v$ is rich.  So $\mu^*(v)\ge 2-\max\{1\cdot2, 1+\frac{1}{2}+\frac{1}{4}\cdot2, \frac{1}{2}\cdot4\}=0$. If $d(v)\ge 5$, then $v$ gives $1$ to each incident $3$-face, at most $\frac{1}{2}$ to each other incident face. So  $$\mu^*(v)\ge2d(v)-6-1\cdot\lfloor\frac{d(v)}{2}\rfloor-\frac{1}{2}\lceil\frac{d(v)}{2}\rceil \ge \frac{5}{4}d(v)-6>0.$$

Next we consider the faces. By (R1), $4$-faces have final charge at least $0$. Note that two $4$-faces in $G$ cannot share a common edge, since this would form a 6-cycle. Thus every $4$-face in $G$ is adjacent to four $7^+$-faces (not necessarily distinct). If $f$ is a $4$-face  with at most two $4^+$-vertices, then by (R1) $\mu^*(f)\ge0$.  Otherwise, by (R2a) $\mu^*(f)\ge 4-6+\min\{1\cdot3-\frac{1}{4}\cdot4, 1\cdot4-\frac{1}{4}\cdot4\}=0$.
Since $G$ contains no $5$- or $6$-faces, we only need to check the final charge of $7^+$-faces.

Let $f$ be a $7^+$-face. By (R1), $f$ only needs to send $\frac{1}{2}x$ to each adjacent $(3,3,3,3)$-face that shares $x$ edges with $f$ and $\frac{1}{4}y$ to each adjacent $(3,3,3,4^+)$-face that shares $y$ edges with $f$ (note that $x,y$ could be more than one since $f$ could contain cut-vertices thus may not be a cycle). Thus $f$ sends at most $\frac{1}{2}$ to all the $4$-faces controlled by a maximal $4$-controlling walk of $f$.

When $d(f)\ge 8$, by using \eqref{s0ti}, we have $\mu^*(f)\ge d(f)-6-\frac{1}{2}\sum_{i=0}^k t_{i}\ge d(f)-6-\frac{1}{4}d(f)=\frac{3}{4}(d(f)-8)\ge 0.$

Let $d(f)=7$. Note that for each element counted in $s_1$, we have either a semi-rich $4$-vertex, from which $f$ gains $\frac{1}{4}$, or a $4$-face with at least two $4^+$-vertices  controlled by a $(3,4^+)$-edge, to which $f$ gives $0$, or a $4$-face with at least three $4^+$-vertices controlled by a $(4^+,4^+)$-edge, from which $f$ gains $\frac{1}{4}$. So by using \eqref{s0ti}, we have
\begin{align*}
\mu^*(f)&\ge d(f)-6-(\frac{1}{2}\sum_{i=0}^kt_{i}-\frac{1}{4}s_1)
\ge d(f)-6-\frac{1}{4}(d(f)-s_0-\sum_{i=0}^ki\cdot t_{i}-s_1)\\
&=\frac{3}{4}(d(f)-8)+\frac{1}{4}(s_0+s_1+\sum_{i=0}^ki\cdot t_{i}).
 \end{align*}

Because $f$ contains no $3$-cycles, $f$ must be a cycle.  By parity, $s_0+\sum_{i=0}^ki\cdot t_{i}\ge1$. Note that $f$ has at most three maximal paths.  We may assume that $f$ has a special $4$-controlling $(3,\ldots,3)$-path, for otherwise $\mu^*(f)\ge7-6-\frac{1}{4}\cdot3>0$. Since $G$ contains no $3,5,6$-cycles, each of the vertices (not on $f$) on $4$-faces adjacent to $f$ has a neighbor not on $f$ or adjacent $4$-faces of $f$. So by   Lemma~\ref{theta-graph-extension}, $f$ contains a $5^+$-vertex,  or a rich $4$-vertex,  or $s_1(f)\ge 2$, or $s_1=1$ and $\sum_{i=2}^k i\cdot t_{i}\ge 2$.  In the first two cases, by (R2a) $f$ gains $\frac{1}{2}$ from the $4^+$-vertex, so $\mu^*(f)\ge\frac{3}{4}(7-8)+\frac{1}{4}+\frac{1}{2}=0$. In the last two cases, $\mu^*(f)\ge \frac{3}{4}(7-8)+\frac{1}{4}\cdot 3=0$.
\end{proof}

\section{Planar graphs without $\{4,5,6,9\}$-cycles, or without $\{4,5,7,9\}$-cycles}

\begin{thm} \label{lem:45p9}
Every planar graph $G$ has no $\{4,5,9,p\}$-cycles with $p \in \{6,7\}$ is DP-$3$-colorable.
\end{thm}

\begin{proof}
Let $f$ be a $10$-face.  Then the following $f$ needs special attention in the discharging(see Figure~\ref{fig:special-10}):
\begin{itemize}
\item $f$ is {\em special} if it either has four special $(3,3)$-paths and one non-special $(3,4)$-path, or has three special $(3,3)$-paths and one special $(3,4,3)$-path and a rich $4$-vertex.

\item $f$ is {\em poor} if it either has four special $(3,3)$-paths and one $(3,5^+)$-path that controls a $(3,3,5^+)$-face, or has three special $(3,3)$-paths and one $(3,4,5^+,3)$-path that controls one $(3,4,5^+)$-face and two $(3,3,4^+)$-faces, or has two special $(3,3)$-paths, one special $(3,4,3)$-path and one $(3,5^+,3)$-path that controls two $(3,3,5^+)$-faces.

\item $f$ is {\em bad} if it has three special $(3,3)$-paths and one $(3,4,4,3)$-path that controls two $(3,3,4)$-faces and one $(4,4,4^+)$-face.
\end{itemize}
\begin{figure}[ht]
\includegraphics[scale=0.15]{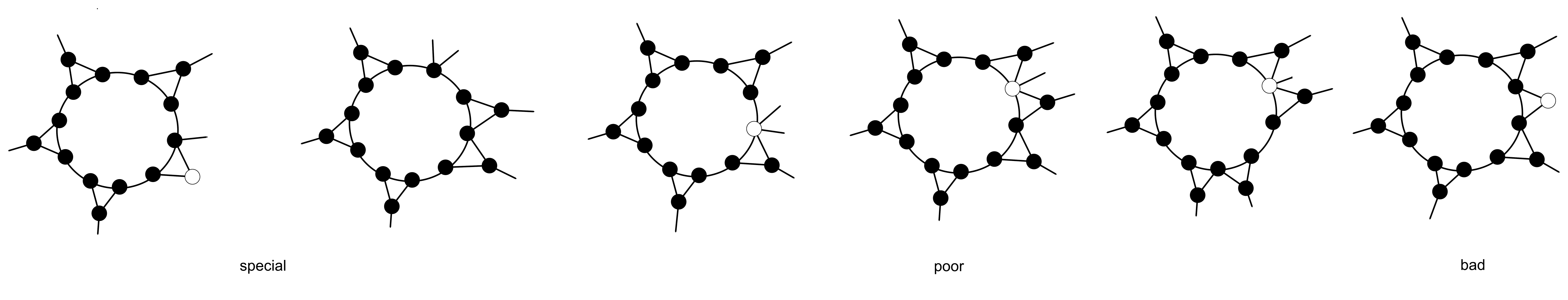}
\caption{The first two are special $10$-faces (the $4$-vertex on $f$ in the second graph could be in other locations), the middle three are poor $10$-faces, and the last one is a bad $10$-face.} \label{fig:special-10}
\end{figure}

Besides (R1), we also have the following discharging rules:
\begin{itemize}
\item[(R2b)] Let $v$ be a $4^+$-vertex on a $10^+$-face $f$.
\begin{itemize}
\item If $d(v)=4$ and $v$ is rich or semi-rich, then $v$ gives $\frac{1}{3}$ to $f$ if $f$ is special, and $\frac{1}{6}$ to $f$ otherwise.

\item If $d(v)=5$, then $v$ gives $\frac{2}{3}$ to $f$ if $f$ is poor, and $\frac{1}{3}$ to $f$ otherwise.
\item If $d(v)\ge6$, then $v$ gives $\frac{2}{3}$ to $f$.
\end{itemize}

\item[(R3b)] After (R1) and (R2b), each vertex and face give its remaining positive charge to the bank, each bad $10$-face gets $\frac{1}{3}$ from the bank.
\end{itemize}

We first check the final charge of vertices. If $d(v)=3$, then $\mu^*(v)=\mu(v)=0$. If $d(v)=4$, then by (R1) and (R2a) $v$ gives $1$ to each incident $3$-face and  at most $\frac{1}{3}$ to each incident $10^+$-face if $v$ is rich or semi-rich. So $\mu^*(f)\ge2\cdot4-6-\max\{1\cdot2, 1+\frac{1}{3}\cdot3, \frac{1}{3}\cdot4\}=0$.  Let $t$ be the number of incident $3$-faces of $v$. If $d(v)\ge5$, then $t\le \lfloor \frac{d(v)}{2} \rfloor$ since 3-faces cannot be adjacent. By (R1)， (R2b)， and (R2c)， $v$ gives $1$ to each incident $3$-face and at most $\frac{2}{3}$ to each incident $10^+$-face. Then
\[ \mu^*(v) \ge 2d(v) - 6 - t-\frac{2}{3}(d(v)-t)=\frac{4}{3}d(v)-6-\frac{1}{3}t\ge \frac{4}{3}d(v)-6-\frac{1}{3}\lfloor \frac{d(v)}{2} \rfloor\ge 0. \]

Next we check the final charge of faces. By (R1), each $3$-face has non-negative final charge.  Since $G$ contains no $4$-cycle, each $7$-face contains no cut-vertices thus must be a cycle. If a planar graph has no cycles of length $\{4,5,6,9\}$, then a 7-face can share at most one edge with $3$-faces. Thus a 7-face gives at most $1$ by (R1), and ends with non-negative charge. In addition, an 8-face cannot be adjacent to any 3-faces and thus gives no charge by (R1). Therefore, in this case, it suffices to show that $10^+$-faces end with non-negative charges. On the other hand, if a planar graph has no cycles of length $\{4,5,7,9\}$, then neither 6-faces nor 8-faces can be adjacent to any 3-faces. Thus, in this case we still only need to make sure that $10^+$-faces end with non-negative charges.

Let $f$ be a $10^+$-face.   If $P$ is a  maximal 3-controlling walk with at least one internal vertex, then $f$ gives at most $\frac{2}{3}$ to the first and last faces controlled by $P$ and $\frac{1}{3}$ to each other face controlled by $P$ (note that one face may appear more than once).
Thus, by (R2b), $f$ gives at most
$$\sum_{i\ge0}\left(\frac{4}{3}+\frac{1}{3}(i-1)\right)t_{i}-\frac{1}{3}s_1=\frac{1}{3}\sum_{i\ge0} (i+3)\cdot t_{i}-\frac{1}{3}s_1.$$

We get
\begin{align*}
\mu^*(f)&\ge d(f)-6-\frac{1}{3}\left(\sum_{i\ge0} (i+3)t_{i}-s_1\right)=\frac{1}{3}\left(2d(f)+s_0+s_1-\sum_{i\ge0} t_{i}\right)-6\\
&\ge \frac{1}{3}\left(2d(f)+s_0+s_1-\frac{1}{2}(d(f)-s_0)+\frac{1}{2}\sum_{i\ge1} (i\cdot t_{i})\right)-6\\
&=\frac{1}{2}(d(f)-12)+\frac{1}{6}\left(3s_0+2s_1+\sum_{i\ge1} i\cdot t_{i}\right).
 \end{align*}

If $d(f) \ge 12$, then this shows $\mu^*(f)\ge 0$.

Let $d(f)\in \{10,11\}$. We may assume that $f$ is cycle. For otherwise, $f$ contains a cut-vertex, then it must be a $10$-face consisting of a $3$-cycle and $7$-cycle, or a $11$-face consisting of a $3$-cycle and $8$-cycle, both of which cannot be adjacent to any $3$-face, thus has nonnegative charge. It follows that all maximal walks on $f$ are paths.

Let $f$ be a $11$-face with $s_0=s_1=0$ or a $10$-face in $G$.  Since $G$ contains no $\{4,5,6,9\}$-cycles or $\{4,5,7,9\}$-cycles, each of the outside vertices on the controlled $3$-faces by edges on $f$ has a neighbor not on $f$ nor on the controlled $3$-faces. As a corollary to Lemma~\ref{theta-graph-extension}, we have the following.

\begin{cor}\label{cor-4569}
Let $f$ be a $10$-face or a $11$-face with $s_0=s_1=0$ in $G$. If $f$ has a special $(3,4,...,4,3)$-path, then $f$ contains a $5^+$-vertex, or a rich $4$-vertex, or $s_1(f)\ge 2$, or $s_1(f)=1$ and $\sum_{i\ge 2} i\cdot t_i\ge 2$.
\end{cor}

Let $d(f)=11$.  Then if $s_0\ge 1$, then $\mu^*(f)\ge\frac{1}{6}\left(2s_1+\sum_{i=1}^k i\cdot t_{i}\right)\ge 0$.  So let $s_0=0$, that is, every vertex of $f$ is on a maximal path. Then by parity, $\sum_{i\ge 1} i\cdot t_i\ge 1$. Now if $s_1\ge 1$, then $\mu^*(f)\ge \frac{1}{6}\left(2s_1+\sum_{i=1}^k i\cdot t_{i}-3\right)\ge 0$. Let $s_1=0$. If $f$ contains a $5^+$-vertex, then $f$ gains at least $\frac{1}{3}$ from the $5^+$-vertex, so $\mu^*(f)\ge \frac{1}{2}(11-12)+\frac{1}{6}\cdot 1+\frac{1}{3}=0$. If not, then $f$ must contain a special $(3,4,...,4,3)$-path, for otherwise, all $3$-faces adjacent to $f$ are $(3,3,4^+)$-faces with the $4^+$-vertex not on $f$. Thus, $\mu^*(f)\ge11-6-\frac{2}{3}\cdot5>0$.


Let $d(f)=10$. Then $$\mu^*(f)\ge \frac{1}{6}\left(3s_0+2s_1+\sum_{i=1}^k i\cdot t_{i}-6\right).$$
So we may assume that $$3s_0+2s_1+\sum_{i\ge 1}i\cdot t_i\le 5.$$ It follows that $s_0\le 1$ and $s_1\le 2$.

First let $s_0=1$. Note that by parity, $\sum_i i\cdot t_i= 1$. So $f$ has one maximal $(3^+,4^+,3^+)$-path and three maximal $(3^+,3^+)$-paths. As $2s_1+\sum_i i\cdot t_i\le 2$, $s_1=0$. By Corollary~\ref{cor-4569}, $f$ either contains a $5^+$-vertex or is a special $10$-face. So by (R2) $f$ gets at least $\frac{1}{3}$ from the incident vertices. So $\mu^*(f)\ge\frac{1}{2}(10-12)+\frac{1}{6}(3+1)+\frac{1}{3}=0$.

Now let $s_0=0$. Then each vertex on $f$ is on a maximal path.

Let $s_1=0$. By Corollary~\ref{cor-4569} $f$ contains at least one poor $5^+$-vertex. So by parity $\sum_{i\ge1}i\cdot t_{i}\ge2$. If $\sum_{i\ge1}i\cdot t_{i}\ge4$, then $\mu^*(f)\ge\frac{1}{6}(4-6)+\frac{1}{3}=0$. Otherwise, by parity $f$ either has three $(3,3)$-paths and one $(3,4^+,4^+,3)$-path or two $(3,3)$-paths and two $(3,4^+,3)$-paths.  So $f$ is either a poor $10$-face or contains two poor $5^+$-vertices. In either case,  $\mu^*(f)\ge\frac{1}{6}(2-6)+\min\{\frac{1}{3}\cdot2, \frac{2}{3}\}=0$.

Let $s_1=1$. First we assume that $f$ contains a $5^+$-vertex. Observe that $\sum_{i\ge 1}i\cdot t_i\le 1$, for otherwise, by (R2b) $\mu^*(f)\ge \frac{1}{6}(2\cdot 1+2-6)+\frac{1}{3}=0$. Thus, by parity $f$ must contain five maximal  $(3^+,3^+)$-paths. Furthermore, $f$ is a poor $10$-face since $s_1=1$. So $\mu^*(f)\ge\frac{1}{6}(2\cdot1-6)+\frac{2}{3}=0$.
Now we consider that $f$ contains no $5^+$-vertices. Since $s_1=1$,  by Corollary~\ref{cor-4569} $f$ has a maximal $(3^+,4^+,...,3^+)$-path with at least two internal vertex. Note that $\sum_{i\ge1}i\cdot t_{i}\le3$. By parity $f$ must have one maximal $(3^+,4^+,4^+,3^+)$-path and three maximal $(3^+,3^+)$-paths, thus have to be a bad $10$-cycle. So $f$ gets $\frac{1}{3}$ from the bank by (R3), and $\mu^*(f)\ge\frac{1}{6}(2+2-6)+\frac{1}{3}=0$.

Let $s_1=2$. Then $\sum_{i\ge 1}i\cdot t_i\le 1$.  Since by parity $f$ cannot just have one maximal path with one internal vertex, $f$ must contain five maximal $(3^+,3^+)$-paths. Since $s_1=2$, $f$ has at least three special $(3,3)$-paths. By Corollary~\ref{cor-4569}, $f$ either contains a $5^+$-vertex, or has two semi-rich $4$-vertices, or is a special $10$-face.  In either case,  $\mu^*(f)\ge \frac{1}{6}(2\cdot2-6)+\min\{\frac{1}{3}, \frac{1}{6}\cdot2\}=0$.

Therefore, all $10$-faces have non-negative final charges.

\medskip

\begin{lem}\label{reduce-4569}
Let $f$ be a bad $10$-face in $G$ and $f_0=[uvw]$ be the $3$-face with three poor $4$-vertices controlled by the $(4,4)$-edge $uv$ on $f$. Let $f'$ be the 10-face containing the edge $vw$ (or by symmetry $uw$). If $s_0=0$ and the only $4^+$-vertices of $f'$ are  poor $4$-vertices, then none of the following holds:
\begin{itemize}
\item[(i)] $\sum_{i\ge1} i\cdot t_i(f')=2$ and $f'$ has a  special (3,3)-path.

\item[(ii)] $\sum_{i\ge1} i\cdot t_i(f')=4$ and $s_1(f')=1$.
\end{itemize}
\end{lem}

\begin{proof}
 Let the  vertices of $f'$ be $v_1, v_2, \ldots, w,v, v'$ in the cyclic order and the vertices on $f$ be $u, u_1, \ldots, u_8, v$ in the cyclic order  so that $vv'u_8$ be a triangle shared by $f$ and $f'$. Let $u'$ be the $3$-vertex adjacent to $u$ and $u_1$ and let $w'$ be the $3$-vertex adjacent to $w$ and $v_7$.

For (i), let $v_iv_{i+1}$ be a special $(3,3)$-path of $f'$ and let $z$ be the $3$-vertex adjacent to $v_i$ and $v_{i+1}$. We order the vertices on $f$ and $f'$ as follows:
$$
v_{i+1},\ v_{i+2},\ \ldots,\ v_7,\ w,\ v,\ u,\ u',\ u_1,\ \ldots,\ u_8,\ v',\ v_1,\ \ldots,\ v_i,\ z
$$

For (ii),  We order the vertices on $f$ and $f'$ as follows:
$$
w,\ v,\ u,\ u',\ u_1,\ \ldots,\ u_8,\ v',\ v_1,\ \ldots,\ v_7,\ w'
$$

If two consecutive vertices of $f'$ are on a special path and at least one of the two vertices has degree four, then we insert the third vertex on the controlled triangle between them. Let $S$ be the set of vertices on our final list. One can easily check that the first vertex has no neighbor outside of $S$ and the last vertex  has a neighbor in $G-S$ and each other vertex in $S$ has at most two neighbors in $G-S$ and earlier vertices in $S$.  Then by Lemma~\ref{near-2-degenerate}, a DP-3-coloring of $G-S$ can be extended to  $S$, a contradiction.
\end{proof}

Finally, we complete the proof by confirming  that the bank has non-negative final charge. Note that the bank only gives a charge $\frac{1}{3}$ to the bad $10$-face by (R3b). Let $f$ be a bad $10$-face in $G$ contains no $5^+$-verices. Let $f_0=[uvw]$ be the $(4,4,4)$-face adjacent to $f$ with the poor $4$-vertex $v$ not on $f$ and $f_i$ be the other two faces adjacent to $f_0$ for $i=1,2$.

\textbf{Case 1:} $d(v)\ge6$. Then $v$ gives at most $1\cdot\lfloor \frac{d(v)}{2} \rfloor+\frac{2}{3}\lceil \frac{d(v)}{2} \rceil\le\frac{5}{6}d(v)$ to all the incident faces. So $v$ gives at least $\frac{7}{6}d(v)-6=\frac{1}{6}d(v)+(d(v)-6)\ge\frac{1}{6}d(v)$ to the bank. Since there are at most $\lfloor \frac{d(v)}{2} \rfloor$ bad $10$-faces adjacent to the $3$-faces which is incident with $v$, the bank gives at most $\frac{1}{3}\lfloor \frac{d(v)}{2} \rfloor\le\frac{1}{6}d(v)$ to all the bad $10$-faces at $v$.

\textbf{Case 2:} $d(v)=5$. Then by definition, none of  $f_1$ and $f_2$ is a poor $10$-face. So $v$ gives at most $1\cdot2+\frac{1}{3}\cdot2+\frac{2}{3}=3\frac{1}{3}$ to all the incident faces and thus gives at least $\frac{2}{3}$ to the bank. Since there are at most two bad $10$-faces adjacent to the $3$-faces containing $v$, the bank gives at most $\frac{2}{3}$ to all the bad $10$-faces at $v$.

\textbf{Case 3:} $d(v)=4$ and $v$ is incident with exactly one $3$-face. Then by definition, none of $f_1$ and $f_2$ is a special $10$-face. So $v$ gives at most $1+\frac{1}{6}\cdot2+\frac{1}{3}=1\frac{2}{3}$ to all the incident faces and thus gives at least $\frac{1}{3}$ to the bank. Since there is exactly one bad $10$-face adjacent to the $3$-face which is incident with $v$, the bank gives at most $\frac{1}{3}$ to the bad $10$-face adjacent to the $3$-face containing $v$.

\textbf{Case 4:} $d(v)=4$ and $v$ is incident with two $3$-faces. Then for $f'\in \{f_1,f_2\}$, $\sum_{i\ge1}i\cdot t_{i}(f')\ge2$. By symmetry we only need to show that $f_1$ can give at least $\frac{1}{6}x$ to the bank, where $x\ge1$ denote the number of bad $10$-faces intersecting to $f_1$.  Then $f_1$ and $f_2$ can give $\frac{1}{6}\cdot2$ to the bank which can be regarded as the charge that the bank sends to $f$. Note that $s_1(f')\ge x$ and $\sum_{i\ge 1} i\cdot t_i(f')\ge x+1$.

Since $G$ contains no $\{4,5,6,9\}$-cycles or $\{4,5,7,9\}$-cycles, $f'$ is a $10^+$-face.
If $d(f')\ge11$, then $$\mu^*(f')\ge\frac{1}{2}(d(f')-12)+\frac{1}{6}\left(3s_0+2s_1+\sum_{i=1}^k i\cdot t_{i}\right)\ge\frac{1}{2}(11-12)+\frac{1}{6}(2x+2)=-\frac{1}{6}+\frac{1}{3}x\ge\frac{1}{6}x.$$

Let $d(f')=10$. If $x\ge 3$ or $s_0\ge 1$, then $$\mu^*(f')\ge \frac{1}{6}\left(3s_0+2s_1+\sum_{i=1}^k i\cdot t_{i}-6\right)\ge\frac{1}{6}(3s_0+2x+x+1-6)=\frac{1}{6}x+\frac{1}{6}(3s_0+2x-5)\ge\frac{1}{6}x.$$

So let $s_0=0$ and $x\le 2$. If $x=2$, then by parity $\sum_{i\ge1} i\cdot t_i\ge4$. So $\mu^*(f')\ge\frac{1}{6}(2\cdot2+4-6)=\frac{1}{3}=\frac{1}{6}x$.
So we let $x=1$. By parity, $\sum_{i\ge1} i\cdot t_i\in \{2,4,6\}$. For $\sum_{i\ge1} i\cdot t_i=2$, if all $4^+$-vertices of $f'$ are poor $4$-vertices, then by Lemma~\ref{reduce-4569}(i) $f'$ has no special $(3,3)$-path. So $\mu^*(f')\ge10-6-\frac{2}{3}\cdot3-\frac{2}{3}\cdot2>\frac{1}{6}$. Otherwise, $f'$ contains a semi-rich $4^+$-vertex, which gives $f'$ at least $\frac{1}{6}$ and $s_1\ge2$. So $\mu^*(f')\ge\frac{1}{6}(2\cdot2+2-6)+\frac{1}{6}=\frac{1}{6}$. For $\sum_{i\ge1} i\cdot t_i=4$, if all $4^+$-vertices of $f'$ are poor $4$-vertices, then  by Lemma~\ref{reduce-4569}(ii) $s_1\ge2$. So $\mu^*(f')\ge\frac{1}{6}(2\cdot2+4-6)>\frac{1}{6}$. Otherwise, $f$ gets at least $\frac{1}{6}$ from the incident $4^+$-vertices. So  $\mu^*(f')\ge\frac{1}{6}(2+4-6)+\frac{1}{6}=\frac{1}{6}$. If  $\sum_{i\ge1} i\cdot t_i=6$, then $\mu^*(f')\ge\frac{1}{6}(2+6-6)>\frac{1}{6}$.
\end{proof}

\section{Planar graphs without $5$-, $6$-, $7$-cycles and with distance of triangles at least two}\label{567T}

\begin{thm}\label{567}
Every planar graph without cycles of lengths in $\{5,6,7\}$ and with distance of triangles at least two is DP-$3$-colorable.
\end{thm}

We say a $8$-face $f$ is {\em special} if one of the following applies:
\begin{enumerate}[$(P_1)$]
\item $f$ is adjacent to one $(3,3,3,3)$-face, two $(3,3,3)$-faces and incident to a rich $4$-vertex.
\item $f$ with a rich $4$-vertex has one special $(3,4,3)$-path, two special $(3,3)$-paths and adjacent to two $3$-faces and two $4$-faces.
\item $f$ is adjacent to one $(3,3,3)$-face, two $(3,3,3,3)$-faces and one $(3,4,4^+)$-face controlled by a $(3,4)$-edge.
\item $f$ is adjacent to one $(3,3,3,3)$-face, two $(3,3,3)$-faces, and a 4-face with at least two $4^+$-vertices controlled by a $(3,4)$-edge.
\end{enumerate}

For convenience, if $f$ is a special 8-face described in $P_i$ for $1\le i\le4$, we say that $f$ is $P_i$. Note that the forbidden cycle lengths require each special $8$-face to be bounded by a cycle. See Figure~\ref{im:special8} for examples of each type of special 8-faces.

\begin{figure}[ht]
\includegraphics[scale=0.15]{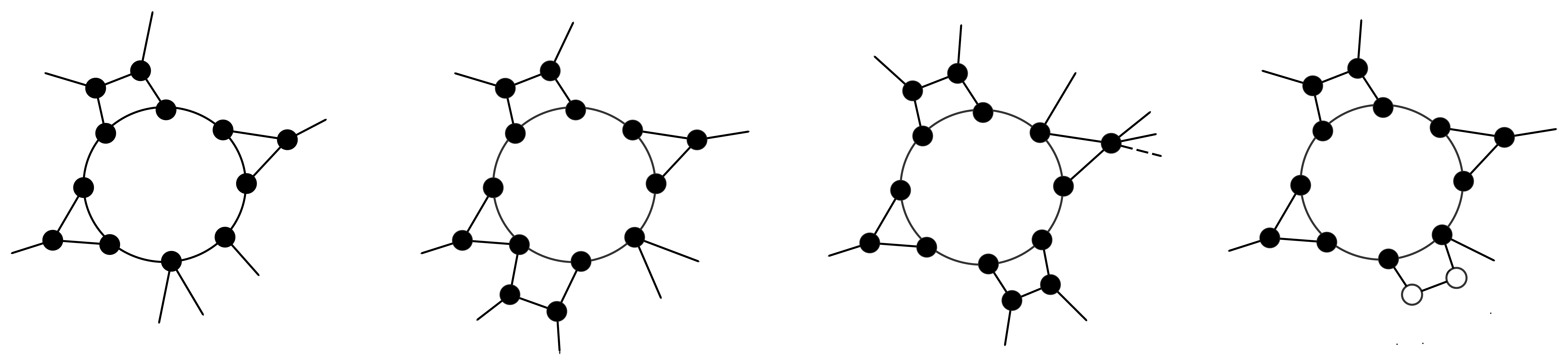}
\caption{Each face shown is special. From left to right, they are in $P_1$, $P_2$, $P_3$, and $P_4$.}\label{im:special8}
\end{figure}

Besides (R1), we also have the following rules:

\begin{itemize}
\item[(R2c)] Every $8$-face receives
\begin{itemize}
\item $1$ from each incident $6^+$-vertex,
\item $\frac{3}{4}$ from incident semi-rich $5$-vertex,
\item $\frac{1}{2}$ from other incident $5$-vertex,
\item $\frac{1}{4}$ from each non-poor $4$-vertex, and
\item $\frac{1}{2}$ from each adjacent $(3^+,4^+,4^+,4^+)$-face that is controlled by a $(4^+, 4^+)$-edge of $f$.
\end{itemize}

\item[(R3c)] After (R1) and (R2c), each vertex and face with positive charge sends its remaining charges to a bank and each special $8$-face gets $\frac{1}{4}$ from the bank.
\end{itemize}

Let $\mu'(x)$ be the charge of $x$ after (R1) and (R2c).

\begin{lem}
Each vertex and $k$-face with $k\ne 8$ in $G$ have non-negative final charge.
\end{lem}

\begin{proof}

By (R3c) we only need to show that each vertex and $k$-face with $k\ne 8$ in $G$ have non-negative charge after (R1) and (R2c).

We first check the final charge of vertices by (R1) and (R2c). If $d(v)=3$, then $\mu'(v)=\mu(v)=0$. If $d(v)=4$, then $v$ gives $1$ to each incident $4^-$-face and $\frac{1}{4}$ to each incident $8$-face when $v$ is incident with at most one $4^-$-face, so $\mu'(v)\ge2-\max\{1\cdot2, 1+\frac{1}{4}\cdot3, \frac{1}{4}\cdot4\}=0$. If $d(v)=5$, then $v$ gives $1$ to each incident $4^-$-face, $\frac{3}{4}$ to each incident $8$-face $f$ when $v$ is semi-rich to $f$, and at most $\frac{1}{2}$ to each other incident face.  Hence, $\mu'(v)\ge4-\max\{1\cdot2+\frac{3}{4}\cdot2+\frac{1}{2}, 1+\frac{3}{4}\cdot4\}=0.$ If $d(v)\ge6$, then $\mu'(v)=2d(v)-6-d(v)\ge0$.

Next we check the final charge of faces. Note that each $4^-$-face can only be adjacent to $7^+$-faces. If $f$ is a $3$-face or a $4$-face with at most two $4^+$-vertices, then by (R1) $\mu'(f)\ge0$. If $f$ is a $4$-face with at least three $4^+$-vertices, then by (R2c) $\mu'(f)\ge4-6+\min\{1\cdot3-\frac{1}{2}\cdot2, 1\cdot4-\frac{1}{2}\cdot4\}=0.$

Note that $G$ contains no $\{5,6, 7\}$-cycles. So a $7$-face in $G$ cannot be a $7$-cycle, thus must contain cut vertices.  Therefore, a $7$-face $f$ is adjacent to at most one $4^-$-face, and when it is, the $4^-$-face must be a $4$-face that contains a $4^+$-vertex and shares two edges with $f$. By (R1), $f$ gives at most $\frac{1}{4}\cdot2$ to the $4$-face. So $\mu'(f)\ge7-6-\frac{1}{2}>0$.  

Let $f$ be a $8^+$-face.  By (R1),  $f$ gives $1$ to each adjacent $(3,3,3)$-face, $\frac{2}{3}$ to each adjacent $(3,3,4^+)$-face, $\frac{1}{3}$ to each adjacent $(3,4^+,4^+)$-face, $\frac{1}{2}$ to each adjacent $(3,3,3,3)$-face, and $\frac{1}{4}$ to each adjacent $(3,3,3,4^+)$-face.

Thus, $f$ gives at most $\max\{1,\frac{2}{3}+\frac{1}{4}, \frac{1}{3}+2\cdot\frac{1}{4}\}\le 1$ to a $(3^+,3^+,3^+)$-face and the two (possible) $4$-faces next to it, and at most $\frac{1}{2}$ to the $4$-faces disjoint from $3$-faces on a maximal $4^-$-controlling walks on $f$.  Let $t_3'$ be the number of $3$-faces adjacent to $f$.  Then
\begin{equation}\label{eq3}
\mu'(f)\ge d(f)-6-1\cdot t_3'-\frac{1}{2}\lfloor\frac{d(f)-s_0-2t_3'}{2}\rfloor= \left(d(f)-\frac{1}{2}\lfloor\frac{d(f)-s_0}{2}\rfloor\right)-6-\frac{1}{2}t_3'.
\end{equation}

Note that $t_3'\le \lfloor\frac{d(f)}{3}\rfloor$. So $\mu'(f)\ge 0$ if $d(f)\ge 10$. For $d(f)=9$, if $t_3'\le 2$, then $\mu'(f)\ge 0$. If $t_3'=3$, then $f$ has no $4$-faces disjoint from the $3$-faces, so $\mu'(f)\ge 9-6-3\cdot 1=0$.
\end{proof}

\begin{lem}\label{8-structure}
Each $8$-face $f$ has $\mu'(f)\ge-\frac{1}{4}$. Furthermore, if $\mu'(f)<0$, then $f$ must be a special $8$-face.
\end{lem}

\begin{proof}
First observe that if $f$ contains a cut vertex, then it must consist of two $4$-cycles. In this case, $f$ is not adjacent to $3$-faces and is adjacent to at most one $4$-face that contains a $4^+$-vertex and shares two edges with $f$. By (R1), $f$ gives at most $\frac{1}{4}\cdot2$ to the $4$-face. So $\mu'(f)\ge8-6-\frac{1}{2}>0$. Therefore, we may assume all $8$-faces are bounded by $8$-cycles.

We may assume that $f$ contains at least three maximal paths. For otherwise, by (R1), $\mu'(f)\ge 8-6-\max\{\frac{4}{3}+\frac{1}{2}, 1\cdot2\}=0$. Thus $s_0\le2$ and by parity each maximal path of $f$ has at most two internal vertices.

We may assume that $f$ contains no $5^+$-vertex. For otherwise, by (R2c),
\begin{itemize}
\item If $f$ contains a $6^+$-vertex, then $\mu'(f)\ge 8-6-1\cdot 2-\frac{1}{2}\cdot 2+1=0$.
\item If $f$ contains a semi-rich $5$-vertex, then $\mu'(f)\ge 8-6-(1\cdot 2+\frac{1}{2}\cdot 2-\frac{1}{4})+\frac{3}{4}=0$.
\item If $f$ contains a rich or poor $5$-vertex, then by parity $f$ must have three maximal paths and $\mu'(f)\ge 8-6-(1\cdot 2+\frac{1}{2})+\frac{1}{2}=0$.
\end{itemize}

We may assume that $f$ contains no rich $4$-vertex. Suppose otherwise. Then $f$ has exactly three maximal paths and gains $\frac{1}{4}$ from the vertex, so $\mu'(f)\ge 8-6-(1\cdot 2+\frac{1}{2})+\frac{1}{4}=-\frac{1}{4}$. Recall that $s_0\le2$. Since the distance of triangles in $G$ is at least two,  when $\mu'(f)<0$, we have (a) $s_0=2$, and $f$ is adjacent to two $(3,3,3)$-faces and one $(3,3,3,3)$-face and contains one rich $4$-vertex and one rich $3$-vertex, which gives us the special face $P_1$, or (b) $s_0=1$, and $f$ is adjacent to two special $(3,3)$-paths and one special $(3,4,3)$-path that controls two $3$-faces and two $4$-faces, which is the special face $P_2$.

We may also assume that each maximal path of $f$ has at most one internal vertices. Suppose  otherwise that $f$ contains a maximal path with two internal $4$-vertices. It implies that $f$ contains two $(3^+,3^+)$-paths and one $(3^+,4,4,3^+)$-path, which control at most two $3$-faces because of the distance condition on triangles. Note that $f$ cannot be adjacent to two $(3,3,3)$-faces. So $f$ gives out at most $1+\frac{1}{2}$ to adjacent $3$-faces and $4$-faces controlled by the two $(3^+,3^+)$-paths and  at most $\frac{2}{3}+\frac{1}{4}<1$ to adjacent $3$-faces and $4$-faces controlled by the $(3^+,4,4,3^+)$-path. Thus $f$ gives out at most $2\frac{1}{2}$. If the $(4,4)$-edge controls a $(4,4,4+)$-face, then $\mu'(f)\ge8-6-1-\frac{1}{2}-\frac{1}{4}\cdot2=0$. If the $(4,4)$-edge controls a $(4,4,4^+,3^+)$-face, then $f$ saves $\frac{1}{2}$ through the $(4,4)$-edge. So we may assume that the $(4,4)$-edge controls a $(4,4,3)$-face or a $(4,4,3,3)$-face.  By Lemma~\ref{theta-graph-extension}, either $f$ contains no special paths, or $s_1(f)\ge 2$. In the former case, $f$ gives out at most $\frac{2}{3}\cdot 2+\frac{1}{4}\cdot 2<2$; in the latter case, $f$ saves at least $\frac{1}{4}\cdot 2=\frac{1}{2}$ in giving.  In either case, $\mu'(f)\ge 0$. 


Since $G$ contains no $5$-, $6$-, or $7$-cycles,  every outside vertex on a controlled $4^-$-face of $f$ must have a neighbor not on $f$ or on controlled $4^-$-faces. We confirm that $f$ must contain a special $(3,3)$-path or a special $(3,4,3)$-path. For otherwise, by (R1), $\mu'(f)\ge 8-6-\max\{\frac{2}{3}\cdot2+\frac{1}{4}\cdot 2\}>0$. Thus Lemma~\ref{theta-graph-extension} implies that $s_1\ge 2$. Note that for each element counted in $s_1$, we have either a semi-rich $4$-vertex, from which $f$ gains $\frac{1}{4}$, or a $(3,4,4^+)$- or $(3,4,4^+,3^+)$-face controlled by $(3,4)$-path, to which $f$ respectively gives at most $\frac{1}{3}$ or $0$. So if $f$ contains three maximal paths, then $\mu'(f)\ge 8-6-(1\cdot 2+\frac{1}{2})+\frac{1}{4}\cdot s_1\ge0$. If $f$ contains four $(3^+,3^+)$-paths, then by Lemma~\ref{theta-graph-extension}(iii) $f$ contains at least one semi-rich $4$-vertex. Now if some $4^-$-face contains two $4^+$-vertices, then $\mu'(f)\ge 8-6-\max\{1+1+\frac{1}{2},1+2\cdot\frac{1}{2}+\frac{1}{3}\}+\frac{1}{4}= -\frac{1}{4}$, and when $\mu'(f)<0$, we have $P_3$ or $P_4$. If every $4^-$-face contains at most one $4^+$-vertex, then $f$ has at least two semi-rich $4$-vertices by Lemma~\ref{theta-graph-extension}(iii), so $\mu'(f)\ge 8-6-(1\cdot 2+\frac{1}{2}\cdot2-\frac{1}{4}\cdot 2)+\frac{1}{4}\cdot2=0$.
\end{proof}

By Lemmas~\ref{8-structure} and (R3c), we guarantee that $8$-faces have non-negative final charge. Thus to finish the contradiction, we must show that the bank has non-negative final charge. This is shown in Lemmas~\ref{reducible2} to~\ref{bank}.

\begin{lem}\label{reducible2}
Let $v$ be a $4$-vertex and $f_i$ for $1\le i\le4$ be the four incident faces of $v$ in the clockwise order. Let $f_1$ be a $4$-face and $f_i$ be a $8^+$-face for $2\le i\le4$. Then each of the following holds:
\begin{enumerate}[(a)]
\item If $f_3$ is $P_1$ or $P_2$, then neither $f_2$ nor $f_4$ is $P_3$ or $P_4$.

\item Let $f_2$ and $f_4$ be $P_4$ and suppose $f_3$ is an 8-face.  If $f_3$ contains a $(3,4,3)$-path, then each of the controlled faces has at least two $4^+$-vertices.

\item Let $f_2$ and $f_4$ be $P_4$ and suppose $f_3$ is an $8$-face. If $f_3$ contains a $(3,3)$-path that controls a $4$-face and another $4$-vertex $u$ other than $v$. Then $u$ cannot be on a $4^-$-face and a $P_4$.
\end{enumerate}
Consequently, $v$ is incident with at most one special $8$-face, unless $v$ is incident with two $P_4$ whose $4$-vertices are on $(3,4,3,4^+)$-faces.
\end{lem}
\begin{center}
\begin{figure}[ht]
\includegraphics[scale=0.28]{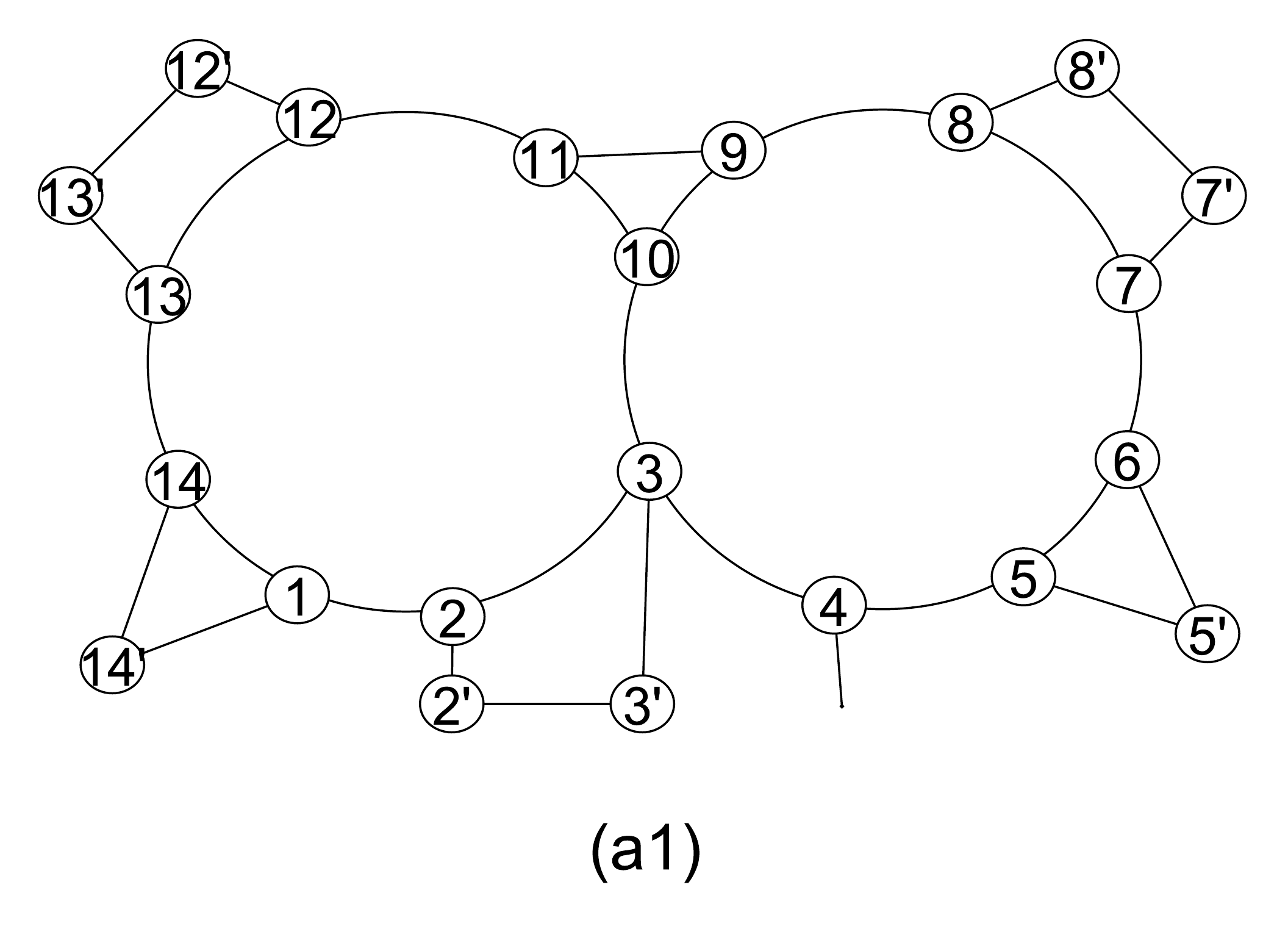}
\includegraphics[scale=0.28]{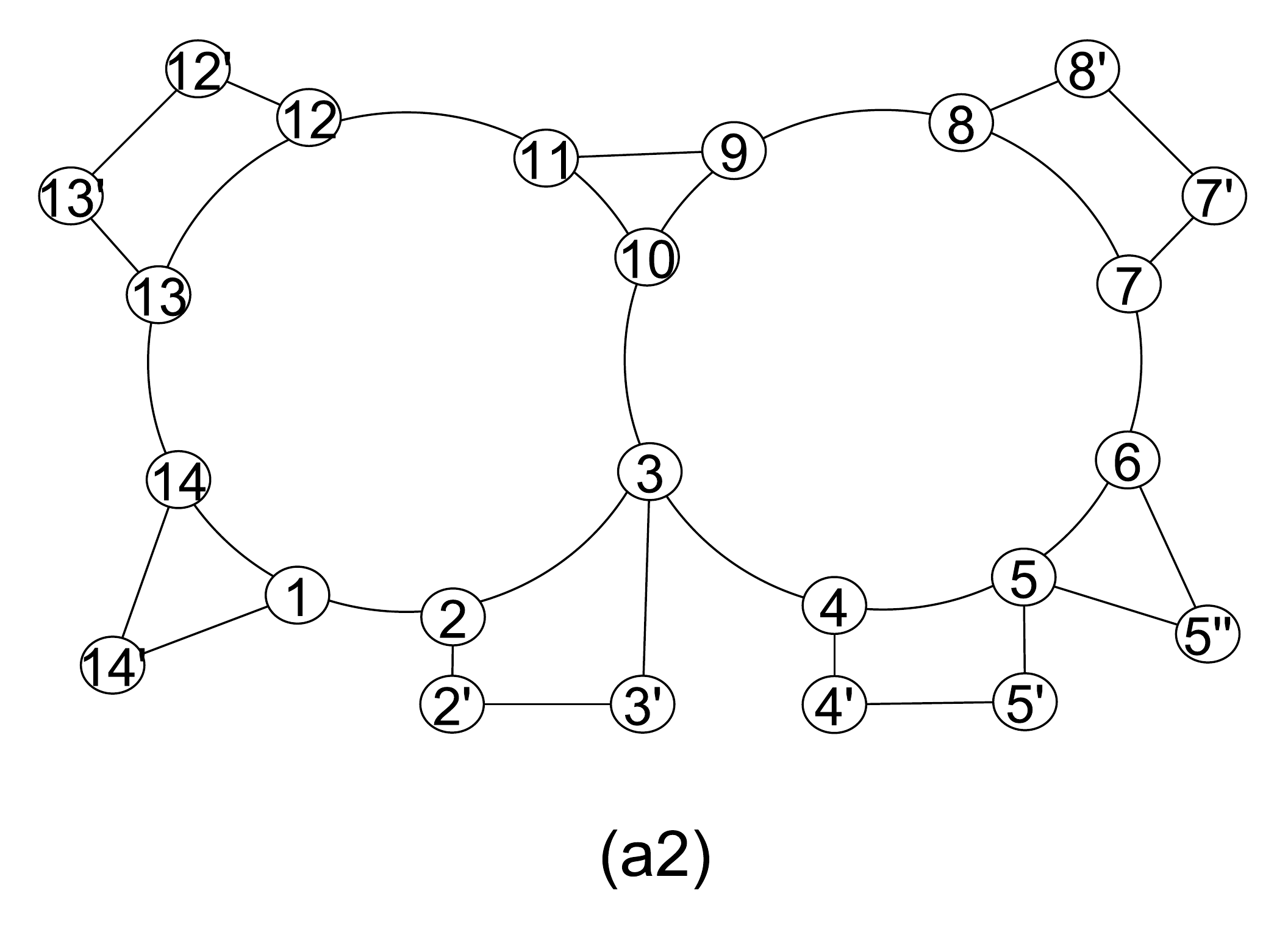}
\includegraphics[scale=0.28]{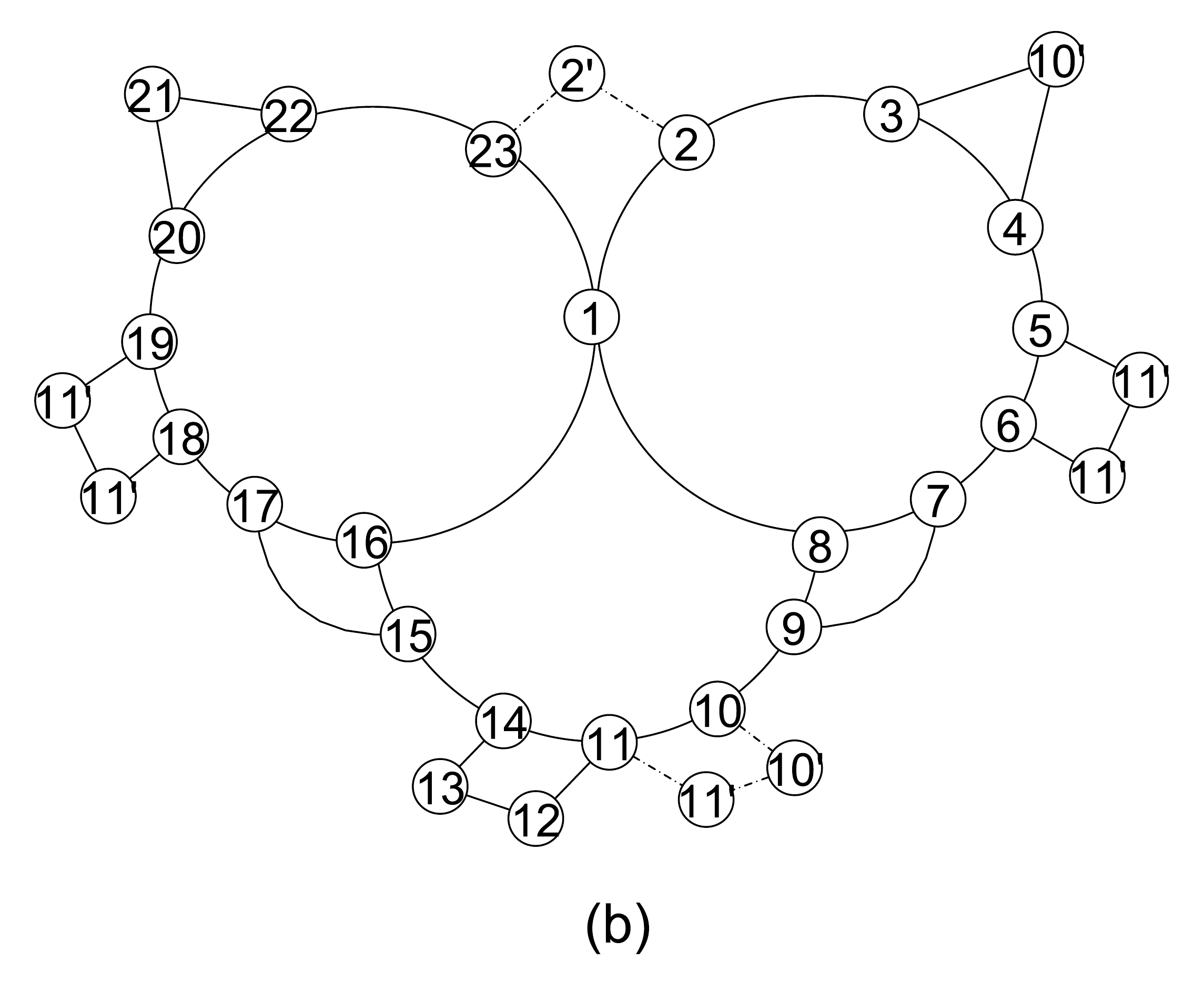}
\includegraphics[scale=0.28]{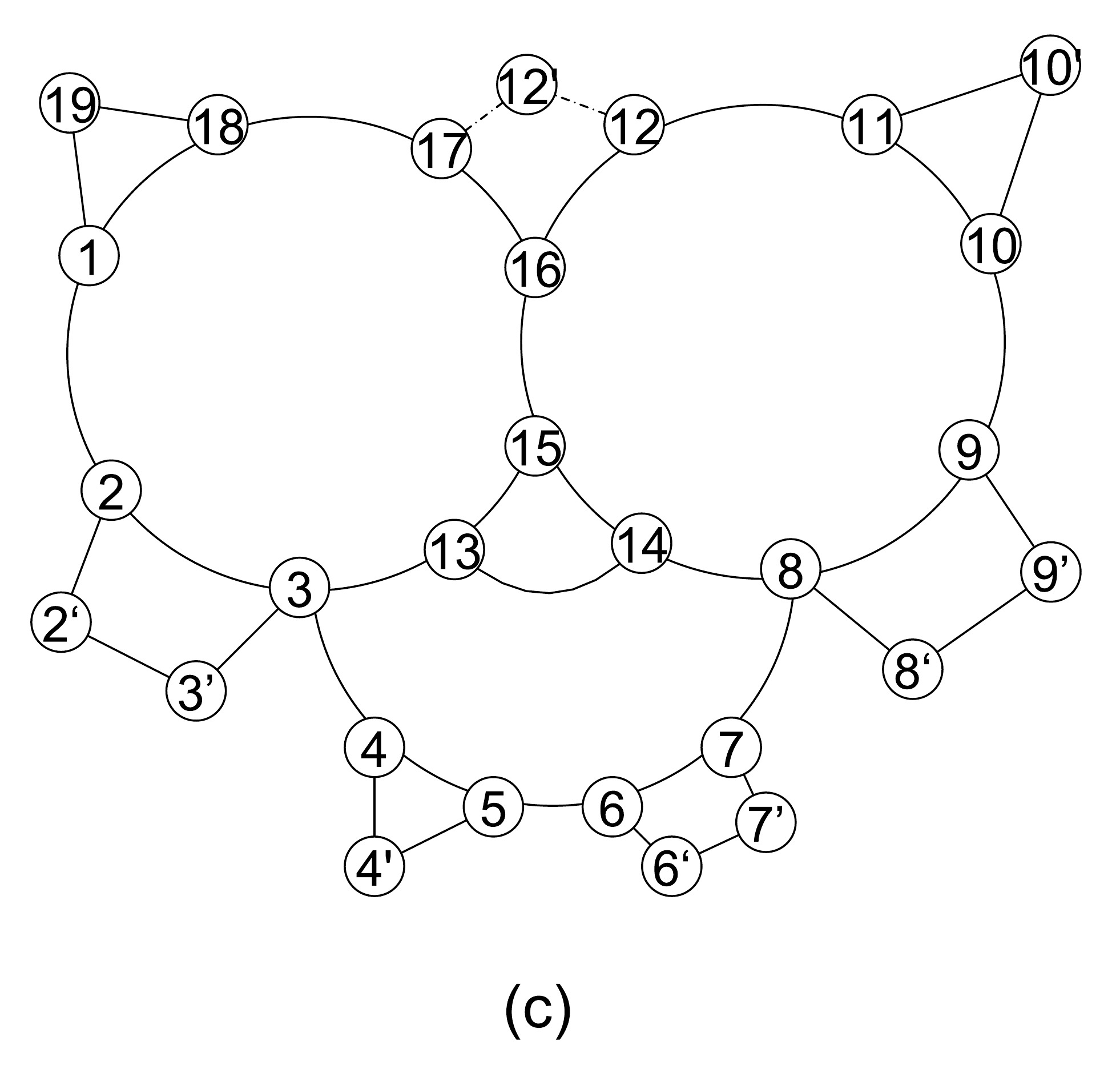}
\caption{The list of vertices in the proof of Lemma~\ref{reducible2} (a) (the first two), (b) and (c).}\label{a}
\end{figure}
\end{center}
\begin{proof}
(a). Let $f_3$ be $P_1$ or $P_2$ and $f_2$ be $P_3$ or $P_4$.  Since $v$ is on a $3$-face in $P_3$ and next to a $3$-face on $P_1$ and $P_2$, $f_2$ must be a $P_4$. Let $vw$ be the common edge of $f_2$ and $f_3$. Let $f_2=vwv_1\ldots v_6$ and $f_3=vww_1w_2\ldots w_6$ in the cyclic orders.  Then $v_1w_1\in E(G)$, and $v_4v_5$ is on a $(3,3,3)$-face, called $v_4v_5u$. Order $u$ and the vertices on $f_2$ and $f_3$ as follows (see Figure~\ref{a}): $$v_5,\ v_6,\ v,\ w_6,\ w_5,\ w_4,\ w_3,\ w_2,\ w_1,\ w,\ v_1,\ v_2,\ v_3,\ v_4,\ u$$
When $f_3$ is $P_2$, we insert the vertices not on $f_3$ but on the $4^-$-faces controlled by the $(3,4,3)$-path to the list in cyclic order.  Let $S$ be the set of vertices in the list. Then it is not hard to check that $S$ has no repeated vertices and $u$ has a neighbor outside of $S$.  By Lemma~\ref{near-2-degenerate}, a DP-3-coloring of $G-S$ can be extended to $G$, a contradiction.

(b). Let $f_2=vv_1\ldots v_7$, $f_3=vv_7v_8\ldots v_{13}$, and $f_4=vv_{13}v_{14}\ldots v_{19}$, so that $f_2,f_4$ are $P_4$s, and $v_9v_{10}v_{11}$ is a $(3,4,3)$-path on $f_3$ that controls two $4$-faces and $v_{10}v_{11}$ is on a $(4,3,3,3)$-face $v_{10}v'v''v_{11}$. Order the vertices in $f_2$, $f_3$,$f_4$ and $v',v''$ as follows (see Figure~\ref{a}): $$v,\ v_2,\ \ldots, v_9,\ v_{10},\ v',\ v'',\ v_{11}, \ \ldots, v_{19}$$
Let $S$ be the set of vertices in the list.  Then it is not hard check that $S$ has no repeated vertices and $v_{19}$ has a neighbor outside of $S$. By Lemma~\ref{near-2-degenerate}, a DP-3-coloring of $G-S$ can be extended to $G$, a contradiction.

(c). Assume that $u$ is on a $4^-$-face and a $P_4$. Since $u$ is next to a triangle, on $f_3$, $u$ is on a $4$-face. Let $f_i'$ for $i\in [4]$ be the four incident faces of $u$ in clockwise order so that $f_1'$ is adjacent to $f_2$. Then $f_2'$ is the $4$-face at $u$ and $f_4'=f_3$. Note that $f_3'$ cannot be $P_4$, since it contains two consecutive $4$-faces. So $f_1'$ is the $P_4$ at $u$.

Follow the labels of vertices on $f_2$ and $f_3$ in (b). Then $u=v_{9}$. Let $f_1'=v_5v_6v_8v_9u_1u_2u_3u_4$.  By definition of $P_4$, $v_2v_3$ is on a $(3,3,3)$-face, say $v_2v_3w_1$, and $u_2u_3$ is also on a $(3,3,3)$-face, say $u_2u_3w_2$. It is not hard to check that $w_1$ has no neighbors on $f_2,f_3,f_1'$ other than $v_2,v_3$.  Order the vertices as follows (see Figure~\ref{a}):
$$v_2, \ v_1,\ v,\ v_{13},\ v_{12},\ v_{11},\ v_{10},\ v_9,\ v_8,\ v_7,\ v_6,\ u_1,\ u_2,\ u_3,\ u_4,\ v_5,\ v_4,\ v_3,\ w_1$$
Let $S$ be the set of vertices in the list. Then it is not hard to check that $S$ has no repeated vertices. By the choice, $w_1$ has a neighbor outside of $S$. By Lemma~\ref{near-2-degenerate}, a DP-3-coloring of $G-S$ can be extended to $G$, a contradiction.
\medskip

Now assume that $v$ is incident with two special $8$-faces. Let $f$ be a special $8$-face at $v$.  Then $v$ is a rich $4$-vertex to $f$ when $f\in \{P_1, P_2\}$ and a semi-rich $4$-vertex to $f$ when $f\in \{P_3, P_4\}$. It follows that when $f\in \{P_1, P_2\}$, then $f$ must be $f_3$, and when $f\in \{P_3, P_4\}$, then $f$ must be $f_2$ or $f_4$.  By (a),  we may assume that $f\in \{P_3, P_4\}$.  Let $f_2=P_3$ by symmetry. It follows that $f_4$ contains a $(4,4^+)$-path, so it cannot be a $P_3$ or $P_4$.  Therefore, we may assume $f_2$ and $f_4$ are both $P_4$.  Then $f_4$ and $f_1$ cannot share a $(4,4^+)$-edge.  So $f_1$ must be a $(3,4,3,4^+)$-face, i.e., the two $4^+$-vertices are not consecutive on $f_1$.
\end{proof}

\begin{lem}\label{bank}
The bank has nonnegative final charge.
\end{lem}

\begin{proof}
By Lemma~\ref{8-structure}, the bank only need to send to special $8$-faces. Let $f$ be a special $8$-face and $v$ be the rich or semi-rich $4$-vertex on $f$. If $v$ is not incident with any $4^-$-faces, then $v$  sends at least $2-\frac{1}{4}\cdot4=1$ to the bank and the bank gives out at most $\frac{1}{4}\cdot4=1$ to all the $8$-faces incident with $v$. So we may assume that $v$ is on a $4^-$-face,  then $v$ sends at least $2-1-\frac{1}{4}\cdot3=\frac{1}{4}$ to the bank. Let  $f_i$ for $1\le i\le4$ be the four incident faces of $v$ in clockwise order and let $f_1$ be the $4^-$-face.

If $v$ is incident with at most one special $8$-face, then the bank can give out the $\frac{1}{4}$ from $v$ to the special $8$-face at $v$.  So we may assume that $v$ is on two special $8$-faces. Then by Lemma~\ref{reducible2},  the special $8$-faces are $f_2$ and $f_4$, and $f_1$ is a $(3,4,3,4^+)$-face. If $f_3$ is a $9^+$-face, then by the rules, $v$ sends at least $2-1-\frac{1}{4}\cdot2=\frac{1}{2}$ to the bank and the bank gives at most $\frac{1}{4}\cdot2=\frac{1}{2}$ to $f_2$ and $f_4$. So we may assume that $f_3$ is a $8$-face.  Note that $v$ is a rich $4$-vertex to $f_3$ and sends at least $\frac{1}{4}$ to the bank. Then we are only need to find $\frac{1}{4}$ in the bank that can be given to the two special $8$-faces.

  Note that $f_3$ is adjacent to two $(3,3,3)$-faces which are also adjacent to $f_2$ and $f_4$, respectively.  For the other three vertices on $f_3$, they may control up to two $4$-faces (and no $3$-face).   So we consider the following three cases:

{\bf Case 1.}  $f_3$ controls no $4$-faces.  In this case, $f_3$ gets at least $\frac{1}{4}$ from $v$ and needs to send out $2$ to the $3$-faces. So by (R3c), $f_3$ gives at least $\frac{1}{4}$, which may be thought of from $v$, to the bank, thus can be used to $f_2$ and $f_4$.

{\bf Case 2.} $f_3$ contains a $(3^+, 4^+, 3^+)$-path that controls two $4$-faces.  In this case, if $f_3$ has a $5^+$-vertex or one semi-rich $4$-vertex, then $\mu^*(f_3)\ge8-6-1\cdot2+\frac{1}{4}-\max\{\frac{1}{4}\cdot2-\frac{1}{2}, \frac{1}{4}-\frac{1}{4}\}=\frac{1}{4}$. By (R3c), $f_3$ gives at least $\frac{1}{4}$ to the bank, which can be used to $f_2$ and $f_4$.
So we may assume that it is a $(3,4,3)$-path. Then by Lemma~\ref{reducible2}(b),  each of the two controlled $4$-faces by the $(3,4,3)$-path contains at least two $4^+$-vertices. So $\mu^*(f_3)\ge8-6-1\cdot2+\frac{1}{4}=\frac{1}{4}>0$. By (R3c), $f_3$ gives at least $\frac{1}{4}$, which again may be thought of from $v$, to the bank, thus can be used to $f_2$ and $f_4$.

{\bf Case 3.}  $f_3$ controls exactly one $4$-face.    If $f_3$ contains a $(3^+,4^+)$-path, then $\mu^*(f_3)\ge8-6-1\cdot2-\frac{1}{4}+\frac{1}{4}\cdot2=\frac{1}{4}>0$. By (R3c), $f_3$ gives at least $\frac{1}{4}$, which may be thought of from $v$, to the bank, thus can be used to $f_2$ and $f_4$.   So we may assume that $f_3$ contains a $(3,3)$-path that controls a $4$-face. Let the eighth vertex of $f_3$ be $u$. The proof of Lemma~\ref{reducible2}(b) implies that $d(u)\ge 4$.

\begin{itemize}
\item  $d(u)\ge5$.  By (R2c), $u$ gives at least $\frac{1}{2}$ to $f_3$. So $\mu^*(f_3)\ge8-6-1\cdot2-\frac{1}{2}+\frac{1}{4}+\frac{1}{2}=\frac{1}{4}>0$. By (R3), $f_3$ gives at least $\frac{1}{4}$, which may be thought of from $v$, to the bank, thus can be used to $f_2$ and $f_4$.

\item $d(u)=4$ and $f_3$ is adjacent to a $(3,3,3^+,4^+)$-face. Then $\mu^*(f_3)\ge8-6-1\cdot2-\frac{1}{4}+\frac{1}{4}\cdot2=\frac{1}{4}>0$. By (R3c), $f_3$ gives at least $\frac{1}{4}$, which may be thought of from $v$, to the bank, thus can be used to $f_2$ and $f_4$.

\item $d(u)=4$ and $f_3$ is adjacent to a $(3,3,3,3)$-face. Note that $u$ is a rich $4$-vertex to $f_3$.  If $u$ is not incident with any $4^-$-faces, then $u$ gives at most $2-\frac{1}{4}\cdot4=1$ to the bank and the bank sends out at most $\frac{1}{4}\cdot3=\frac{3}{4}$ to at most three special $8$-faces at $u$ (note that $f_3$ is not special).  So the bank may send $\frac{1}{4}$ from $u$ to $f_2$ or $f_4$ at $v$.
If $u$ is incident with a $4^-$-face, say $f_1'$, then $f_1'$ must be a $4$-face since the distance of triangles in $G$ is at least $2$.  Let $f_1',f_2',f_3,f_4'$ be the four faces incident with $v$ in clockwise order, where $f_4'$ is adjacent to $f_2$.  Since $f_2'$ is adjacent to two $4$-faces at distance $1$, $f_2'$ cannot be a special $8$-face. Since $f_4'$ contains a semi-rich $4$-vertex next to an adjacent $(3,3,3)$-face, $f_4'$ cannot be $P_1, P_2$ or $P_3$.  By Lemma~\ref{reducible2}(c),  $f_4'$ cannot be $P_4$. So $u$ is not on a special $8$-face.  Then $u$ gives at least $2-1-\frac{1}{4}\cdot3=\frac{1}{4}$ to the bank, which can be given to $f_2$ or $f_4$.
\end{itemize}

So the bank has nonnegative final charge.
\end{proof}

\section*{Acknowledgement}

The authors are very thankful for the referees for their careful reading and many helpful comments.


\begin{thebibliography}{99}
\bibitem{A00}
N. Alon, Degrees and choice numbers, {\em Random Structures $\&$ Algorithms}, 16(2000) 364--368.

\bibitem{B16}
A. Bernshteyn. The asymptotic behavior of the correspondence chromatic number, {\em Discrete Math.}, 339(2016) 2680--2692.

\bibitem{B17}
A. Bernshteyn, The Johansson--Molloy Theorem for DP-Coloring, \emph{arXiv:1708.03843}.

\bibitem{BK17a}
A. Bernshteyn, A. Kostochka, On differences between DP-coloring and list coloring, \emph{arXiv:1705.04883}.

\bibitem{BK17b}
A. Bernshteyn, A. Kostochka, Sharp Dirac's Theorem for DP-Critical Graphs, {\em J. Graph Theory}, 88(2018), 521-546.

\bibitem{BKP17}
A. Bernshteyn, A. Kostochka, S. Pron, On DP-coloring of graphs and multigraphs, {\em Siberian Mathematical Journal}, 58(2017), 28--36

\bibitem{BKZ17}
A. Bernshteyn, A. Kostochka, X. Zhu, DP-colorings of graphs with high chromatic number, {\em European J. of Comb.}, 65(2017), 122-129.

\bibitem{B13}
O. Borodin, Colorings of plane graphs: A survey, {\em Discrete. Math.}, 313(2013), 517–-539.

\bibitem{DP17}
Z. Dvo\v{r}\'{a}k, L. Postle, \emph{Correspondence coloring and its application to list-coloring planar graphs without cycles of length $4$ to $8$}, {\em J. Combin. Theory Ser. B}, 129(2018), 38--54.

\bibitem{ERT79}
P. Erd\H{o}s, A. Rubin, H.Taylor, Choosability in graphs, {\em Congr. Numer.}, 26(1979) 125--157.

\bibitem{G59}
H. Gr\"otzsch, Ein Dreifarbensatz f\"ur Dreikreisfreie Netze auf der Kugel, {\em Math.-Natur. Reihe} 8(1959) 109-120.


\bibitem{KO17a}
S.-J. Kim, K. Ozeki, A note on a Brooks' type theorem for DP-coloring, \emph{arXiv:1709.09807v1}.

\bibitem{KO17b}
S.-J. Kim, K. Ozeki, A sufficient condition for DP-$4$-colorability, {\em Discrete Math.}, 341(2018) 1983--1986.

\bibitem{KY17}
S.-J. Kim, X. Yu, Planar graphs without $4$-cycles adjacent to triangles are DP-$4$-colorable, \emph{arXiv:1712.08999}

\bibitem{LSS05}
C. Lam, W. Shiu, Z. Song,
The $3$-choosability of plane graphs of girth $4$, {\em Discrete Math},  294(2005) 297-–301.

\bibitem{LCW16}
X. Li, M. Chen, and Y. Wang, On $3$-choosability of planar graphs without $5$-, $6$- or $7$-cycles.{\em Adv. Math.} (China) 45(2016) 491-–499.

\bibitem{L09}
B. Lidick\'{y}, On $3$-choosability of plane graphs having no $3$-, $6$-, $7$- and $8$-cycles, {\em Australasian Journal of Combinatorics}, 44(2009) 77-–86.

\bibitem{LLNSY18}
R. Liu, X. Li, K. Nakprasit, P. Sittitrai, G. Yu, DP-$4$-colorability of planar graphs without given two adjacent cycles, submitted.

\bibitem{T94}
C. Thomassen, Every planar graph is $5$-choosable, {\em J. Combin. Theory Ser. B}, 62(1994) 180--181.

\bibitem{T95}
C. Thomassen, $3$-list-coloring planar graphs of girth $5$, {\em J. Combin. Theory Ser. B}, 64(1995) 101--107.


\bibitem{V76}
V. Vizing, Vertex colorings with given colors, {\em Metody Diskret. Analiz, Novosibirsk}, 29(1976) 3-10(in Russian).

\bibitem{V93}
M. Voigt, List coloring of planar graphs, {\em Discrete Math.,} 120(1993) 215--219.

\bibitem{V95}
M. Voigt, A not $3$-choosable planar graph without $3$-cycles, {\em Discrete Math.,} 146(1995) 325--328.


\bibitem{ZW04}
L. Zhang, B. Wu, Three-choosable planar graphs without certain small cycles, {\em Graph Theory Notes of New York}, 46(2004) 27–-30.

\bibitem{ZW05a}
L. Zhang, B. Wu, A note on $3$-choosability of planar graphs without certain cycles, {\em Discrete Math}, 297(2005) 206--209.



\end{thebibliography}
\end{document}